\documentclass[final,onefignum,onetabnum]{siamart190516}

\usepackage{amsmath,amscd,amssymb}
\usepackage{xspace}
\usepackage{cite,alltt}
\usepackage{enumitem}
\setlist[itemize]{topsep=2pt,parsep=0pt,partopsep=2pt}
\setlist[enumerate]{topsep=2pt,parsep=0pt,partopsep=2pt}

\usepackage{mathtools, color}
\usepackage{multirow}
\usepackage{rotating}
\usepackage{subfigure}
\usepackage{epsfig}
\usepackage{url}
\usepackage{graphicx}
\usepackage{mathdots}
\usepackage[english]{babel}
\usepackage{csquotes}
\usepackage{mathrsfs}
\usepackage{overpic}
\usepackage{algorithmic} %algorithm package needed for example 1
\usepackage{algorithm}
\usepackage{blindtext}
\usepackage{mathrsfs}
\usepackage{hyperref}
\newcommand{\ep}{\varepsilon}

\newcommand{\vc}{\mathbf{c}}
\newcommand{\vd}{\mathbf{d}}
\newcommand{\ve}{\mathbf{e}}

\newcommand{\vn}{\mathbf{n}}

\newcommand{\vs}{\mathbf{s}}
\newcommand{\vst}{\widetilde{\mathbf{s}}}

\newcommand{\vu}{\mathbf{u}}

\newcommand{\vx}{\mathbf{x}}
\newcommand{\vy}{\mathbf{y}}
\newcommand{\vxn}{\mathbf{x}}

\newcommand{\vom}{\boldsymbol{\xi}}

\newcommand{\vxi}{\boldsymbol{\xi}}

\newcommand{\prad}{\rho}

\newcommand{\vH}{\mathbf{H}}

\newcommand{\tpsi}{\widetilde{\psi}}
\newcommand{\vI}{\mathbf{I}}

\newcommand{\rot}{\mathbf{L}}
\newcommand{\sgrad}{\mathbf{G}}

\newcommand{\M}{\mathcal{P}}

\newcommand{\Phidiv}{\Phi_{\rm div}}
\newcommand{\Phidivt}{\widetilde{\Phi}_{\rm div}}
\newcommand{\Phicurl}{\Phi_{\rm curl}}

\newcommand{\R}{\mathbb{R}}
\newcommand{\sphere}{\mathbb{S}}
\newcommand{\dist}{\textnormal{dist}}
\newcommand{\sdist}{\textnormal{dist}_{\surface}}

\definecolor{matred}{rgb}{0.6350, 0.0780, 0.1840}
\definecolor{matblue}{rgb}{0, 0.4470, 0.7410}
\definecolor{matpurp}{rgb}{0.4940, 0.1840, 0.5560}
\definecolor{matgreen}{rgb}{0, 0.5, 0}

\newcommand{\dfpuapprox}{\widetilde{\vs}}
\newcommand{\cfpot}{\varphi}
\newcommand{\dfpot}{\psi}

\newcommand{\localadjdfpot}{\widetilde{\dfpot}}

\newcommand{\blendedpotential}{\widetilde{\dfpot}}

\newcommand{\gluept}{\bar{\vx}}
\newcommand{\patch}{\Omega}
\newcommand{\glueptset}{\bar{X}}
\newcommand{\shift}{b}
\newcommand{\rhs}{c}
\newcommand{\res}{r}
\newcommand{\nodeset}{X}

\newcommand{\surface}{\mathcal{P}}
\newcommand{\surfacegrad}{\mathbf{G}}
\newcommand{\surfacecurl}{\mathbf{L}}

\newcommand{\bigO}{\mathcal{O}}
\newcommand{\ds}{\displaystyle}

\newcommand{\revision}[1]{{#1}}

\newcommand{\vertiii}[1]{{\left\vert\kern-0.25ex\left\vert\kern-0.25ex\left\vert #1 \right\vert\kern-0.25ex\right\vert\kern-0.25ex\right\vert}}

\newsiamremark{remark}{Remark}
\newsiamremark{hypothesis}{Hypothesis}
\crefname{hypothesis}{Hypothesis}{Hypotheses}
\newsiamthm{claim}{Claim}

\vspace{-3cm}
\headers{A PUM for div-free or curl-free RBF approximation}{K.~P. Drake, E.~J. Fuselier, and G.~B. Wright}
\title{A Partition of Unity Method for Divergence-free or Curl-free Radial Basis Function Approximation}
\author{Kathryn P. Drake\thanks{Department of Mathematics, Boise State University, Boise, ID (\email{KathrynDrake@u.boisestate.edu}, \email{gradywright@boisestate.edu})}
\and Edward J. Fuselier\thanks{Department of Mathematics, High Point University, High Point, NC (\email{efuselie@highpoint.edu})}
\and Grady B. Wright\footnotemark[1]}

\begin{document}
\maketitle

\begin{abstract}
Divergence-free (div-free) and curl-free vector fields are pervasive in many areas of science and engineering, from fluid dynamics to electromagnetism. A common problem that arises in applications is that of constructing smooth approximants to these vector fields and/or their potentials based only on discrete samples. Additionally, it is often necessary that the vector approximants preserve the div-free or curl-free properties of the field to maintain certain physical constraints. Div/curl-free radial basis functions (RBFs) are a particularly good choice for this application as they are meshfree and analytically satisfy the div-free or curl-free property. However, this method can be computationally expensive due to its global nature. In this paper, we develop a technique for bypassing this issue that combines div/curl-free RBFs in a partition of unity framework, where one solves for local approximants over subsets of the global samples and then blends them together to form a div-free or curl-free global approximant. The method is applicable to div/curl-free vector fields in $\R^2$ and tangential fields on two-dimensional surfaces, such as the sphere, and the curl-free method can be generalized to vector fields in $\R^d$. The method also produces an approximant for the scalar potential of the underlying sampled field.  We present error estimates and demonstrate the effectiveness of the method on several test problems.
\end{abstract}



\section{Introduction}
Approximating vector fields from scattered samples is a pervasive problem in many scientific applications, including, for example, fluid dynamics, meteorology, magnetohydrodynamics, electromagnetics, gravitational lensing, imaging, and computer graphics.  Often these vector fields have certain differential invariant properties related to an underlying physical principle.  For example, in incompressible fluid dynamics the velocity of the fluid is divergence-free (div-free) as a consequence of the conservation of mass.  Similarly, in electromagnetics the electric field is curl-free in the absence of a time varying magnetic field as a consequence of the conservation of energy.  Additionally, the fields may have properties of being tangential to a surface (e.g.,\ the sphere $\sphere^2$) and have a corresponding surface div-free or curl-free property, as occurs in many areas of geophysical sciences~\cite{FPLM18}.  In several of these applications it is necessary for the approximants to preserve these differential invariants to maintain certain physical constraints.  For example, in incompressible flow simulations using the immersed boundary method, excessive volume loss can occur if the approximated velocity field of the fluid is not div-free~\cite{BAO2017183}.

To enforce these differential invariants on the approximant, one cannot approximate the individual components of the field separately, but must combine them in a particular way.  One idea uses the property that div-free fields (in two dimensions) and curl-free fields can be defined in terms of a scalar potential (e.g., a stream function or electric potential).  These methods then compute an approximant for the potential of the field by solving a Poisson equation involving the divergence or curl of the sampled field~\cite{Bhatia2013}.  A separate idea is to use a vector basis for the approximant that satisfies the underlying differential invariant.  This paper develops a radial basis function (RBF) method that uses latter approach, but has similarities to the former.

RBFs are a main tool for scattered data approximation~\cite{MeshFree_Mat,WendlandScatteredData,FFBook}.  In the early 1990s, researchers began to focus on the problem of developing vector RBF interpolants for div-free fields that analytically satisfy the div-free constraint~\cite{AmodeiBenbourhim_1991,Handscomb_1992_solenoidalTPS,NarcowichWard1994:GenHermite}.  The idea, as presented in~\cite{NarcowichWard1994:GenHermite}, is to use linear combinations of shifts of a matrix-valued kernel, whose columns satisfy the div-free property, to interpolate the samples of given field.  Since these kernels are constructed from scalar-valued RBFs, they are referred to as div-free RBFs.  These ideas were later extended to curl-free fields in~\cite{DoduRabut_2004_DFCFInterp,Fuselier2008:VectorError}.  Further extensions of the idea to vector fields tangential to a two-dimensional surface (e.g.,\ $\sphere^2$) that are surface div-free or curl-free were given in~\cite{NarcWardWright}.  Some applications of these div/curl-free RBFs can, for example, be found in~\cite{lowitzsch05-3,HLWHAE2012,Wendland2009:DivFreeStokes,shankar_2018,CoeEtAl2008,FUSELIER201641,MITRANO201594}.

There are, however, issues with scaling div/curl-free RBF interpolation to large data sets.  For a data set with $N$ scattered nodes $\nodeset=\{\vx_j\}_{j=1}^N$, the method requires solving a $dN$-by-$dN$ linear system, where $d=2,3$ is the dimension of the underlying domain.  Additionally, each evaluation of the resulting interpolant involves $dN$ terms.  If the div/curl-free RBFs are constructed from scalar-valued RBFs with global support, then the linear system is dense and not well suited to iterative methods.  To ameliorate these issues, a multilevel framework has been developed for compactly supported div/curl-free RBFs in~\cite{Farrell2016}.  However, we take a different approach to reducing the computational cost using the partition of unity method (PUM)~\cite{LazzaroMontefusco,wendland_2002,MeshFree_Mat,cavoretto_2010,Larsson2017}.

In RBF-PUM, one only needs to solve for local approximants over small subsets of the global data set and then blend them together to form a smooth global approximant.  A particular challenge with extending this idea to div/curl-free RBFs is in enforcing that the global approximant is analytically div/curl-free.  To overcome this challenge, we use the local div/curl-free RBFs to obtain local approximants to scalar potentials for the field and then blend these together to form a global scalar potential for the entire field.  A div/curl-free vector approximant is then obtained by applying the appropriate differential operator to the global scalar potential. The method as presented here will only work for fields that can be defined by scalar potentials, which includes div/curl-free vector fields in $\R^2$, surface div/curl-free tangential fields on two-dimensional surfaces, and curl-free fields in $\R^d$, but not div-free fields in $\R^3$.  However, there are several benefits of the method.  First, for node sets $X$ that are quasiuniform, the algorithm parameters can be chosen to produce global approximants to the field in $\bigO(N\log N)$ operations.  Second, we have error estimates showing the method can give high rates of convergence, and numerical evidence that rates faster than algebraic with increasing $N$ are possible.  Unlike the method from~\cite{Farrell2016}, these convergence rates are possible with the fixed complexity of $\bigO(N\log N)$.  Finally, a global approximant for the scalar potential is given directly from the samples without having to compute derivatives of the sampled field or solving a Poisson problem.  

As far as we are aware, the only other computationally scalable div-free approximation technique for scattered data is the div-free moving least squares (MLS) method from~\cite{TRASK2018310}.  The method is used for generating finite difference type discretizations for Stokes' equations.  While it worked quite successfully for this application, it can be computationally expensive for more general approximation problems, since it requires solving a new (small) linear system for each evaluation point. For the method we develop, the (small) linear systems are independent of the evaluation points.  \revision{Additionally, the div-free MLS method does not directly allow the potential for the field to also be approximated.}

The rest of the paper is organized as follows.  In the next section we introduce some background material necessary for the presentation of the method.  Section 3 contains a review of PUM  and then presents the div/curl-free RBF-PUM.  Error estimates for the new method are presented in Section 4.  Section 5 contains numerical experiments demonstrating the convergence rates of the method on three model problems.
%: a div-free field in a star-shaped domain in $\R^2$, a surface div-free field on the $\sphere^2$, and a curl-free field in the three dimensional unit ball. 
The final section contains some concluding remarks.

\section{Div/Curl-free RBFs}\label{sec:review}
%A na\"ive meshfree approach for reconstructing a vector field from (possibly scattered) samples is to use a separate meshfree scalar-valued approximant for each component of the vector field.  However, if the underlying field is div-free or curl-free, then neither of these properties can generally be preserved with this technique. Instead, one must approximate all of the components of the vector field simultaneously.

We review the generalized vector RBF techniques for reconstructing vector fields below, first for div-free fields and then for curl-free fields.  In both cases, we focus on approximations of tangential vector fields on smooth, orientable, surfaces embedded in $\R^3$ (which includes $\R^2$ and $\sphere^2$).  In the curl-free case the method extends trivially to $\R^d$.  Before discussing these two techniques, we introduce some notation and review some relevant background material.

\subsection{Notation and preliminaries}
Let $\M$ denote a smooth, orientable surface embedded in $\R^3$, possibly with a boundary, and let $\vn\in\R^3$ denote the unit normal vector to $\M$ expressed in the Cartesian basis.  When discussing tangential vector fields on $\M$, we use the terms divergence and curl to be tacitly understood to refer to surface divergence and surface curl for $\M$.  The surface curl (or \textit{rot}) operator $\rot$ and the surface gradient operator $\sgrad$ play a central role in defining div-free and curl-free tangential fields on $\M$.  We can express these operators in extrinsic (Cartesian) coordinates as follows:
%\begin{align*}
%\text{Surface curl:} & &\rot &= \vn \times \nabla, \\
%\text{Surface gradient:} & &\sgrad &= (I - \vn\vn^T) \nabla,
%\end{align*}
\[\rot = \vn \times \nabla,\quad\quad \sgrad = (I - \vn\vn^T)\nabla,\]
where $\nabla$ is the standard $\R^3$ gradient, and $I$ is the $3$-by-$3$ identity matrix. It is \revision{a well known consequence of Poincar\'e's Lemma} that div-free and curl-free fields are locally images of these operators\revision{\cite{DoCarmo}\footnote{Poincar\'e's Lemma is typically given in terms of the exterior derivatve operator $d$. In this case applying the \emph{Hodge star} operator * to $\mathbf{u}$ before applying Poincar\'e's Lemma gives the div-free result. For the curl-free result, one starts with $*d\vu = 0$ and applying the Hodge star operator to this allows one to apply the lemma.}} 

\begin{proposition}\label{thm:scalar_pot2d}
Let $\vu$ be a tangential vector field defined on $\M$ then 
\begin{enumerate}
\item $\vu$ is div-free iff \revision{for each point $\vx\in \M$ there exists a neighborhood $U\subset \M$ and a scalar potential $\dfpot:U\longrightarrow\R$} such that $\vu=\rot(\dfpot)$
\item $\vu$ is curl-free \revision{for each point $\vx\in \M$ there exists a neighborhood $U\subset \M$ and a scalar potential $\cfpot:U\longrightarrow\R$} such that $\vu=\sgrad(\cfpot)$
\end{enumerate}
\end{proposition}
\noindent \revision{Note that since $\rot$ and $\sgrad$ only annihilate constant functions along $\M$, the scalar potentials are unique up to the addition of a constant. 

  The present method relies on this property as it solves for scalar potentials on overlapping patches covering the domain of interest.  Since each of these potentials is unique up to a constant, a straightforward procedure can be derived to determine these values so that the potentials can be shifted to agree over the domain.  In three dimensions, div-free vector fields have vector potentials unique up to the addition of the gradient of a harmonic scalar function, and it not clear to us how to adapt the current method to this situation.}  However, the method will be applicable to curl-free fields in higher dimensions since a vector field $\vu$ on $\R^d$ is curl-free if and only if $\vu = \nabla\cfpot$ for some scalar potential. 

In what proceeds, we use the following notation for the $\rot$ operator:
\begin{align}
\rot = 
\underbrace{
\begin{bmatrix}
0 & -a_3 & a_2 \\
a_3 & 0 & -a_1 \\
-a_2 & a_1 & 0
\end{bmatrix}}_{\ds Q_{\vx}}
\nabla,
\label{eq:rot}
\end{align}
where $\vn = (a_1,a_2,a_3)$ is the unit normal to $\M$ at $\vx$.  Note that applying $Q_{\vx}$ to a vector in $\R^3$ gives the cross product of $\vn$ with that vector.  Similarly, we express $\sgrad$ as
\begin{align}
\sgrad = P_{\vx}\nabla,
\label{eq:sgrad}
\end{align}
where $P_{\vx} = \vI - \vn\vn^T$ projects any vector at $\vx$ on $\M$ into a plane tangent to $\M$ at $\vx$.

Two important cases of $\M$ are $\M=\R^2$ and $\M=\sphere^2$.  For the former case, the unit normal is independent of its position and is typically chosen as $\vn = (0,0,1)$.  Using this with \eqref{eq:rot} and \eqref{eq:sgrad}, leads to the standard definition for these operators for vector fields on $\R^2$:
\begin{align}
\rot = \begin{bmatrix} -\partial_y \\ \partial_x \\ 0 \end{bmatrix}\; \text{and}\; \sgrad = \begin{bmatrix} \partial_x \\ \partial_y \\ 0 \end{bmatrix},  \label{eq:rot2d}
\end{align}
which can be truncated to remove the unnecessary third component.  For $\M=\sphere^2$, the unit normal at $\vx$ is $\vn=\vx$, but $\rot$ and $\sgrad$ do not simplify beyond this.

%For a vector field $\vu$ defined in $\R^3$, the existence of a scalar potential that defines the field similar to Theorem \ref{thm:scalar_pot2d} is only available when $\vu$ is a curl-free field and not when it is div-free (here the terms curl and divergence refer to the these operators defined in $\R^3$).  The following is a consequence of the Helmholtz decomposition theorem~\cite{Bhatia2013}
%\begin{theorem}\label{thm:scalar_pot3d}
%Let $\vu$ be a sufficiently smooth vector field in $\R^3$ then the following hold:
%\begin{enumerate}
%\item If $\vu$ is div-free then there exists a vector potential $\dfvpot:\R^3\longrightarrow\R^3$ such that $\vu=\nabla\times\dfvpot$.
%\item If $\vu$ is curl-free then there exists a scalar potential $\cfpot:\R^3\longrightarrow\R$ such that $\vu=\nabla \cfpot$.
%\end{enumerate}
%In the case second case, the potential $\cfpot$ is unique up to the addition of a constant.
%\end{theorem}
%The existence of a scalar potential and its uniqueness up to a constant is what allows us to develop the partition of unity method for curl-free fields also in $\R^3$ as discussed in Section \ref{sec:pum}.

\subsection{Div-free RBF interpolation}
Div-free vector RBF interpolants are similar to scalar RBF interpolants in the sense that one constructs them from linear combinations of shifts of a kernel at each of the given data sites. The difference between the approaches is that in the vector case one uses a \emph{matrix-valued kernel} whose columns are div-free.  For the sake of brevity, we give the final construction of these kernels and \revision{refer the reader} to~\cite{NarcWardWright} for a rigorous derivation. For more information on scalar-valued RBFs, which we do not discuss here, see any of the books~\cite{MeshFree_Mat,WendlandScatteredData,FFBook}.

Let $\phi:\R^3\times\R^3\longrightarrow\R$ be a radial kernel in the sense that $\phi(\vx,\vy)=\eta(\|\vx-\vy\|)$, for some $\eta:[0,\infty)\longrightarrow\R$, where $\|\cdot\|$ is the vector $2$-norm.  It is common in this case to simply write $\phi(\vx,\vy) = \phi(\|\vx-\vy\|)$. Supposing $\phi$ has two continuous derivatives, then the matrix kernel $\Phidiv$ is constructed using the operator $\rot$ in \eqref{eq:rot} as
\revision{
\begin{equation}
\begin{aligned}
		\Phidiv(\vx,\vy) = \rot_{\vx}^{\phantom{T}} \rot_{\vy}^{T}\phi\left(\|\vx-\vy\|\right) &= 
		Q_{\vx}\left(\nabla_{\vx}^{\phantom{T}} \nabla_{\vy}^T\phi\left(\|\vx-\vy\|\right)\right)Q_{\vy}^T \\ 
		& = Q_{\vx}\left(\nabla\nabla^T\phi\left(\|\vx-\vy\|\right)\right)Q_{\vy},
\end{aligned}
\label{eq:phidvi}
\end{equation}
}
where the subscripts in the differential operators indicate which variables they operate on and, for simplicity, no subscript means they operate on the $\vx$ component. Here we have used the fact that the matrix $Q_{\vy}$ in \eqref{eq:rot} is skew-symmetric and  $\nabla_{\vy}^T\phi\left(\|\vx-\vy\|\right) = -\nabla^T\phi\left(\|\vx-\vy\|\right)$.  For any $\vc\in\R^3$ and fixed $\vy\in\M$, the vector field $\Phidiv(\vx,\vy)\vc$ is tangent to $\M$ and div-free in $\vx$, which follows from Proposition \ref{thm:scalar_pot2d} since
\begin{align}
\Phidiv(\vx,\vy)\vc = Q_{\vx}\nabla \left(\nabla^T\phi\left(\|\vx-\vy\|\right)Q_{\vy}\vc\right) = \rot(\dfpot(\vx)),
\label{eq:PhiDivPot}
\end{align}
\revision{where $\dfpot$ is the potential for $\Phidiv(\vx,\vy)\vc$.}
The second argument of $\Phidiv$ acts as a shift of the kernel and indicates where the field $\Phidiv\vc$ is ``centered.''

An interpolant to a div-free tangential vector field $\vu:\M\longrightarrow\R^3$ sampled at distinct points $\nodeset=\{\vxn_j\}_{j=1}^N\subset\M$ can be constructed using $\Phidiv$ as follows:
\begin{equation}
	\vs(\vx)=\sum_{j=1}^N\Phidiv(\vx,\vxn_j)\vc_j,
	\label{eq:FullDivInterp}
\end{equation}
where the coefficients $\vc_j \in \R^3$ are tangent to $\M$ at $\vxn_j$ (this is necessary to make the interpolation problem well-posed as discussed below) and are chosen so that $\vs\bigr|_{\nodeset} = \vu\bigr|_{\nodeset}$.  We refer to \eqref{eq:FullDivInterp} as a div-free RBF interpolant.  

Instinctively, one may try to solve for the expansion coefficients in \eqref{eq:FullDivInterp} by imposing  $\vs(\vxn_j) = \vu_j$, $j=1,\ldots,N$, where $\vu_j = \vu(\vxn_j)$.  However, this will lead to a singular system of equations since each $\vu_j$ can be expressed using only two degrees of freedom rather than three.  To remedy this, let $\{\vd_j,\ve_j,\vn_j\}$ be orthonormal vectors at the node $\vxn_j$, where  $\vn_j$ is the outward normal to $\M$, $\ve_j$ is a unit tangent vector to $\M$, and $\vd_j=\vn_j\times \ve_j$.  Since $\vu_j$ is tangent to $\M$ we can write it in this basis as $
\vu_j = \gamma_j \vd_j+\delta_j\ve_j$, where $\gamma_j=\vd_j^T\vu_j$ and $\delta_j=\ve_j^T\vu_j$.
%Thus, there are only $2N$ values to fit, but $3N$ coefficients in \eqref{eq:FullDivInterp}.  We can make these the same value by enforcing that each
We may also express each tangent $\vc_j$ as $\vc_j = \alpha_j\vd_j+\beta_j\ve_j$, which leads us to  express \eqref{eq:FullDivInterp} as % where $\alpha_j  = \vd_j^T\vc_j$ and $\beta_j = \ve_j^T\vc_j$ are the new coefficients to determine.
\begin{align}
  %	\vs(\vx)=\sum_{j=1}^N\Phidiv(\vx,\vxn_j)\underbrace{\left[\alpha_j\vd_j+\beta_j\ve_j\right]}_{\ds \mathbf{c}_j},
	\vs(\vx)=\sum_{j=1}^N\Phidiv(\vx,\vxn_j)\left[\alpha_j\vd_j+\beta_j\ve_j\right],
	\label{eq:rbfvsinterp}
\end{align}
and to write the interpolation conditions as $\vd_i^T\vs(\vxn_i) = \gamma_i$ and  $\ve_i^T\vs(\vxn_i) = \delta_i$.  This leads to the $2N$-by-$2N$ system of equations
\begin{align}
	\sum_{j=1}^N\underbrace{\left(\begin{bmatrix}\vd_i^T\\
			\ve_i^T\end{bmatrix}\Phidiv(\vxn_i,\vxn_j)\begin{bmatrix}\vd_j &
			\ve_j\end{bmatrix}\right)}_{\ds A^{(i,j)}}\begin{bmatrix}\alpha_j\\
		\beta_j\end{bmatrix}=\begin{bmatrix}\gamma_i\\
		\delta_i\end{bmatrix}, \quad 1\leq i \leq N.
	\label{eq:divfree_linsys}
\end{align}
The interpolation matrix that arises from this system (with its $(i,j)^{\rm th}$ $2$-by-$2$ block given by $A^{(i,j)}$) is positive definite if $\Phidiv$ is constructed from an appropriately chosen scalar-valued RBF (e.g., a positive definite $\phi$)~\cite{NarcWardWright}.

%\begin{table}[h!]
%	\centering
%	\begin{tabular}{|c|c|}
%		\hline
%		\textbf{Radial kernel} & \textbf{Expression, $\phi(r)$} \\ \hline
%		Multiquadric (MQ)               & $(1 + (\ep r)^2)^{\frac12}$ \\ \hline
%		Inverse multiquadric (IMQ) & $(1 + (\ep r)^2)^{-\frac12}$ \\ \hline
%		Gaussian (GA)                   & $\exp(-(\ep r)^2)$ \\ \hline
%		Mat\'ern (MA)                     & $\exp(-\ep r)\left(1 + (\ep r) + \frac{5}{7}(\ep r)^2 + \frac{2}{21}(\ep r)^3 + (\ep r)^5\right)$ \\ \hline
%	\end{tabular}
%	\caption{\label{tbl:rbfs} Examples of radial kernels that result in positive definite div-free and curl-free systems.  Here $\ep>0$ is the shape parameter.}
%\end{table}

When $\M=\R^2$, the div-free RBF interpolant can be simplified considerably since in this case we can choose $\vd_j = (1,0,0)$ and $\ve_j = (0,1,0)$ and use \eqref{eq:rot2d} for defining $\Phidiv$.  Using this in \eqref{eq:rbfvsinterp} and truncating the unnecessary third component of the vector interpolant (since it is always zero) gives the expansion
\begin{align}
\tilde{\vs}(\vx) = \sum_{j=1}^N \Phidivt(\vx,\vxn_j)\tilde{\vc}_j,
\label{eq:2dDivInterp}
\end{align}
where $\tilde{\vs},\tilde{\vc}_j\in\R^2$,  and 
\begin{align*}
\Phidivt(\vx,\vxn_j) = \begin{bmatrix}-\partial_{yy} & \partial_{xy} \\ \partial_{xy} & -\partial_{xx} \end{bmatrix}\phi(\|\vx-\vxn_j\|).
\end{align*}
This expression for $\Phidivt$ can be written as $\Phidivt = -I\Delta \phi + \nabla\nabla^T\phi$, which is the standard way to express div-free kernels for general $\R^d$~\cite{Fuselier2008:VectorError}.

%We conclude by noting that once the coefficients $\vc_j$ are determined in \eqref{eq:FullDivInterp}, we can extract a scalar  for the field using \eqref{eq:PhiDivPot}:
An important consequence from the construction of the div-free RBF interpolant \eqref{eq:FullDivInterp} is that we can extract a scalar potential $\dfpot$ for the interpolated field.  Using \eqref{eq:PhiDivPot} for $\Phidiv$ in \eqref{eq:FullDivInterp} we have
\begin{align}
\vs(\vx)=\sum_{j=1}^N\Phidiv(\vx,\vxn_j)\vc_j = \underbrace{Q_{\vx}\nabla}_{\ds \rot}\biggl(\underbrace{\sum_{j=1}^N\nabla^T\phi\left(\|\vx-\vxn_j\|\right)Q_{\vx_j}\vc_j}_{\ds \dfpot(\vx)}\biggr) = \rot(\dfpot(\vx)).
\label{eq:rbf_sf}
\end{align}
\revision{This potential will play a crucial role in developing the PUM in Section \ref{sec:pum}}.

\subsection{Curl-free RBF interpolation}
Curl-free vector RBF interpolants are constructed in a similar fashion to the div-free ones, the only difference being that $\sgrad$ is applied instead of $\rot$ in the construction of the matrix kernel. Given a scalar RBF $\phi$ and using a derivation similar to \eqref{eq:phidvi}, $\Phicurl$ is given as
\begin{equation}
\begin{aligned}
		\Phicurl(\vx,\vy) = \sgrad_{\vx} \sgrad_{\vy}^{T}\phi\left(\|\vx-\vy\|\right) &= -P_{\vx}\left(\nabla \nabla^T\phi\left(\|\vx-\vy\|\right)\right)P_{\vy},
\end{aligned}
\end{equation}
where we have used the fact that the $P_{\vx}$ matrix in \eqref{eq:sgrad} is symmetric.  For any $\vc\in\R^3$ and fixed $\vy\in\M$, the vector field $\Phicurl(\vx,\vy)\vc$ is tangential to $\M$ and curl-free in $\vx$.  This follows from Proposition \ref{thm:scalar_pot2d} since
\begin{align}
\Phicurl(\vx,\vy)\vc = P_{\vx}\nabla \left(-\nabla^T\phi\left(\|\vx-\vy\|\right)P_{\vy}\vc\right) = \sgrad(\cfpot(\vx)),
\label{eq:PhiCurlPot}
\end{align}
\revision{where $\cfpot$ is the potential for $\Phicurl(\vx,\vy)\vc$.}  As with the div-free kernel \eqref{eq:PhiDivPot}, the second argument of $\Phicurl$ acts as a shift of the kernel and indicates where the field $\Phicurl\vc$ is ``centered''.

Interpolants to a curl-free tangential vector field $\vu:\M\longrightarrow\R^3$ sampled at distinct points $X=\{\vxn_j\}_{j=1}^N\subset\M$ are constructed from $\Phicurl$ as 
\begin{equation}
	\vs(\vx)=\sum_{j=1}^N\Phicurl(\vx,\vxn_j)\vc_j,
	\label{eq:FullCurlInterp}
\end{equation}
where the coefficients $\vc_j \in \R^3$ are tangent to $\M$ at $\vxn_j$ and are chosen so that $\vs\bigr|_\nodeset = \vu\bigr|_\nodeset$. The procedure for determining these coefficients is identical to the div-free case, one just needs to replace $\Phidiv$ with $\Phicurl$ in \eqref{eq:rbfvsinterp} \& \eqref{eq:divfree_linsys}.  The matrix from the linear system \eqref{eq:divfree_linsys} with $\Phicurl$ is similarly positive definite for the same $\phi$. Further, a scalar potential $\cfpot$ can also be extracted from the curl-free field \eqref{eq:FullCurlInterp} using \eqref{eq:PhiCurlPot}:
%~\cite{NarcWardWright}.
\begin{align}
\vs(\vx)= \underbrace{P_{\vx}\nabla}_{\ds \sgrad}\biggl(\underbrace{-\sum_{j=1}^N\nabla^T\phi\left(\|\vx-\vxn_j\|\right)P_{\vxn_j}\vc_j}_{\ds \cfpot(\vx)}\biggr) = \sgrad(\cfpot(\vx)).
\label{eq:rbf_vp}
\end{align}
%where we have used \eqref{eq:PhiCurlPot} for $\Phicurl$ in \eqref{eq:FullCurlInterp}.

%Rather than give a special treatment for curl-free RBF interpolants when $\M=\R^2$ as we did for the div-free ones, we describe them together with the case of curl-free fields in $\R^3$ since in both cases a scalar potential can be extracted (see Theorem \ref{thm:scalar_pot3d}).
In the Euclidean case $\R^d$, the curl-free kernel is simply given as 
$\Phicurl(\vx,\vy) = -\nabla\nabla^T \phi(\|\vx-\vy\|)$~\cite{Fuselier2008:VectorError}, 
%This is just the Hessian of $\phi$, so that in $\R^2$, $\Phicurl$ is a $2$-by-$2$ matrix-valued function, while in $\R^3$ it is $3$-by-$3$.
where $\nabla$ is the $d$-dimensional gradient. The interpolation conditions $\vs\bigr|_X = \vu\bigr|_X$ also lead to the simplified linear system for the expansion coefficients $\vc_j\in\R^d$:
\begin{align}
\sum_{j=1}^N\Phicurl(\vxn_i,\vxn_j)\vc_j = \vu_i,\; i=1,2,\ldots,N,
\label{eq:curlfree_linsys}
\end{align}
which is $dN$-by-$dN$.  A scalar potential $\cfpot$ for the vector interpolant can be extracted as
\begin{align}
\vs(\vx)= \nabla\biggl(\underbrace{-\sum_{j=1}^N\nabla^T\phi\left(\|\vx-\vxn_j\|\right)\vc_j}_{\ds \cfpot(\vx)}\biggr).
\label{eq:CurlInterp_Rd}
\end{align}

\section{A div-free/curl-free partition of unity method}\label{sec:pum}
%One issue with div-free and curl-free RBF interpolants is the high computational cost associated with solving the linear systems \eqref{eq:divfree_linsys} and \eqref{eq:curlfree_linsys} for the interpolation coefficients when the number of nodes $N$ in $\nodeset$ is large.
%This requires $\mathcal{O}(N^3)$ operations for a direct method when a globally supported RBF is used.
The cost associated with solving the linear systems \eqref{eq:divfree_linsys} and \eqref{eq:curlfree_linsys} is $\mathcal{O}(N^3)$, which is prohibitively high when the number of nodes $N$ in $\nodeset$ is large. In this section, we develop a partition of unity method (PUM) that requires solving several linear systems associated with subsets $X_\ell$ of $\nodeset$ with $n_\ell << N$ nodes, which reduces the computational cost significantly regardless of the nature of the RBF used.

\subsection{Partition of unity methods}
Let $\Omega\subset\R^d$ be an open, bounded domain of interest for approximating some function $f:\Omega\longrightarrow\R$. Let $\Omega_1,\dots,\Omega_M$ be a collection of distinct overlapping patches that form an open cover of $\Omega$, i.e., $\cup_{\ell=1}^M \Omega_\ell \supseteq \Omega$, and let the overlap between patches be limited such that at most $K << M$ patches overlap at any given point $\vx\in\Omega$.  For each $\ell=1,\ldots,M$, let $w_\ell:\Omega_\ell\longrightarrow[0,1]$ be a weight function such that $w_\ell$ is compactly supported on $\Omega_\ell$ and the set of weight functions $\{w_\ell\}$ have the property that $\sum_{\ell=1}^{M} w_\ell \equiv 1$.
%\begin{align}
%\sum_{\ell=1}^{M} w_\ell \equiv 1.
%\label{eq:pu_propery}
%\end{align}
Suppose $s_\ell$ is some approximation to $f$ on each patch $\Omega_\ell$. The partition of unity approach of Babu\v{s}ka and Melenk~\cite{Babuska_PU} is to form an approximant $s$ to $f$ over the whole domain $\Omega$ by ``blending'' the local approximants $s_\ell$ with $w_\ell$ via $s = \sum_{\ell=1}^M w_\ell s_\ell$.
%\begin{align}
%
%\label{eq:basic_pum}
%\end{align}

When samples of $f$ are given at $N$ ``scattered'' nodes $\nodeset=\{\vxn_j\}_{j=1}^N\subset\Omega$, RBF interpolants are a natural choice for the local approximants $s_\ell$, as pointed out in~\cite{Babuska_PU}.  RBF-PUM was first explored for interpolation in 2002 by Wendland~\cite{wendland_2002} and Lazzaro and Montefusco~\cite{LazzaroMontefusco}, and then later in 2007 by Fasshauer~\cite[Ch.\ 29]{MeshFree_Mat}.  More recent work has explored various aspects of the method in terms of applications, methods, and implementations, especially by Cavoretto, De Rossi, and colleagues (e.g.,~\cite{cavoretto_15,cavoretto_2019,CaDeRoPe:2015}), and also extensions to problems on the sphere~\cite{cavoretto_2010,shankar_2018}.  Additionally, the method has been adapted for approximating the solution of partial differential equations (e.g.,~\cite{Safdari-Vaighani2015,Shcherbakov2016,Larsson2017,KevinThesis}).

Common choices for the patches in RBF-PUM are disks for problems in $\R^2$, spherical caps for problems on $\mathbb{S}^2$, and balls for problems in $\R^3$, and these are the choices we use throughout this paper.  Figure \ref{fig:RBF_PU} gives an example of a set of patches for a problem in $\R^2$.  Techniques for choosing the patches are discussed in, e.g.,~\cite{CaDeRoPe:2015,Larsson2017,shankar_2018} (see Section \ref{sec:implementation} for more discussion).  \revision{Other choices for patches commonly used in PUM methods are rectangles and procedures for generating these can be found, for example, in~\cite{GriebelSchweitzer2002}.}

Based on the choices of patches, the weight functions $w_\ell$ can be constructed using Shepard's method as follows.  Let $\kappa:\mathbb{R}^{+}\rightarrow\mathbb{R}$ have compact support over the interval $[0,1)$.  For each patch $\Omega_\ell$, let $\vom_\ell$ denote its center and $\prad_\ell$ denote its radius, and define $\kappa_\ell(\vx) := \kappa\left(\|\vx-\vom_\ell\|/\prad_\ell\right)$. The weight functions are then given by
%\begin{align}
%\kappa_\ell(\vx) := \kappa\left(\frac{\|\vx-\vom_\ell\|}{\prad_\ell}\right).
%\label{eq:wght_func}
%\end{align}
\begin{equation}
w_\ell(\vx) = \kappa_\ell(\vx)/\ds \sum_{j=1}^{M} \kappa_j(\vx),\; \ell=1,\ldots,M.\nonumber
\end{equation}
Note that each $w_\ell$ is only supported over $\Omega_\ell$ and that the summation on the bottom only involves terms that are non-zero over patch $\Omega_\ell$, which is bounded by $K$.
Figure~\ref{fig:RBF_PU} (b) illustrates one of these weights functions for the example domain in part (a), where $\kappa$ is chosen as the $C^1$ quadratic $B$-spline
\begin{align}
\kappa(r) = \begin{cases} 1 - 3r^2, & 0\leq r \leq \frac{1}{3}, \\ \frac{3}{2}(1 - r)^2, & \frac{1}{3}\leq r \leq 1. \end{cases}
\end{align}
This is the weight function we use throughout the paper.
\begin{figure}[h]
	\centering
	\begin{tabular}{cc}
	\includegraphics[width=0.4\textwidth]{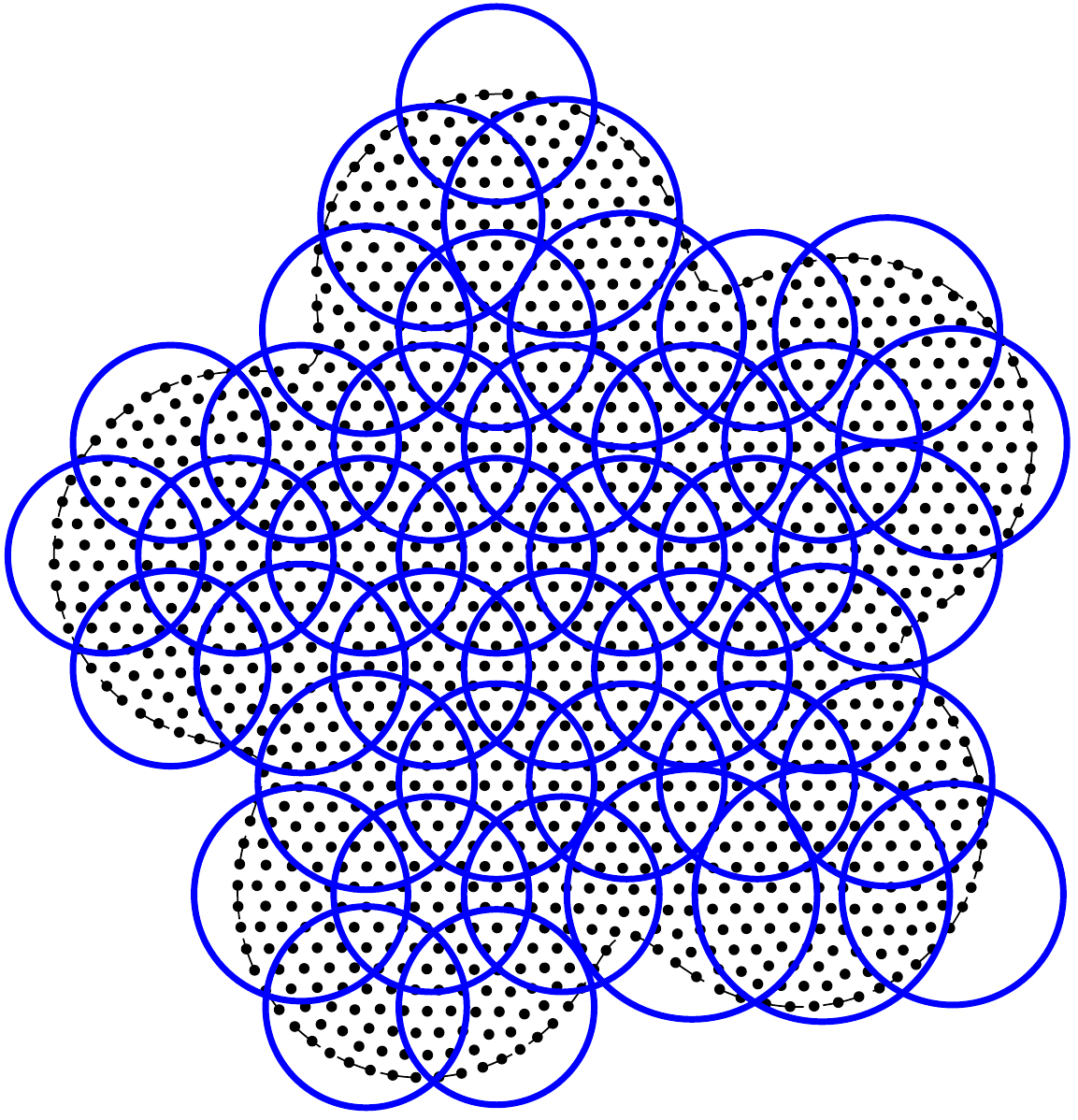} & \includegraphics[width=0.4\textwidth]{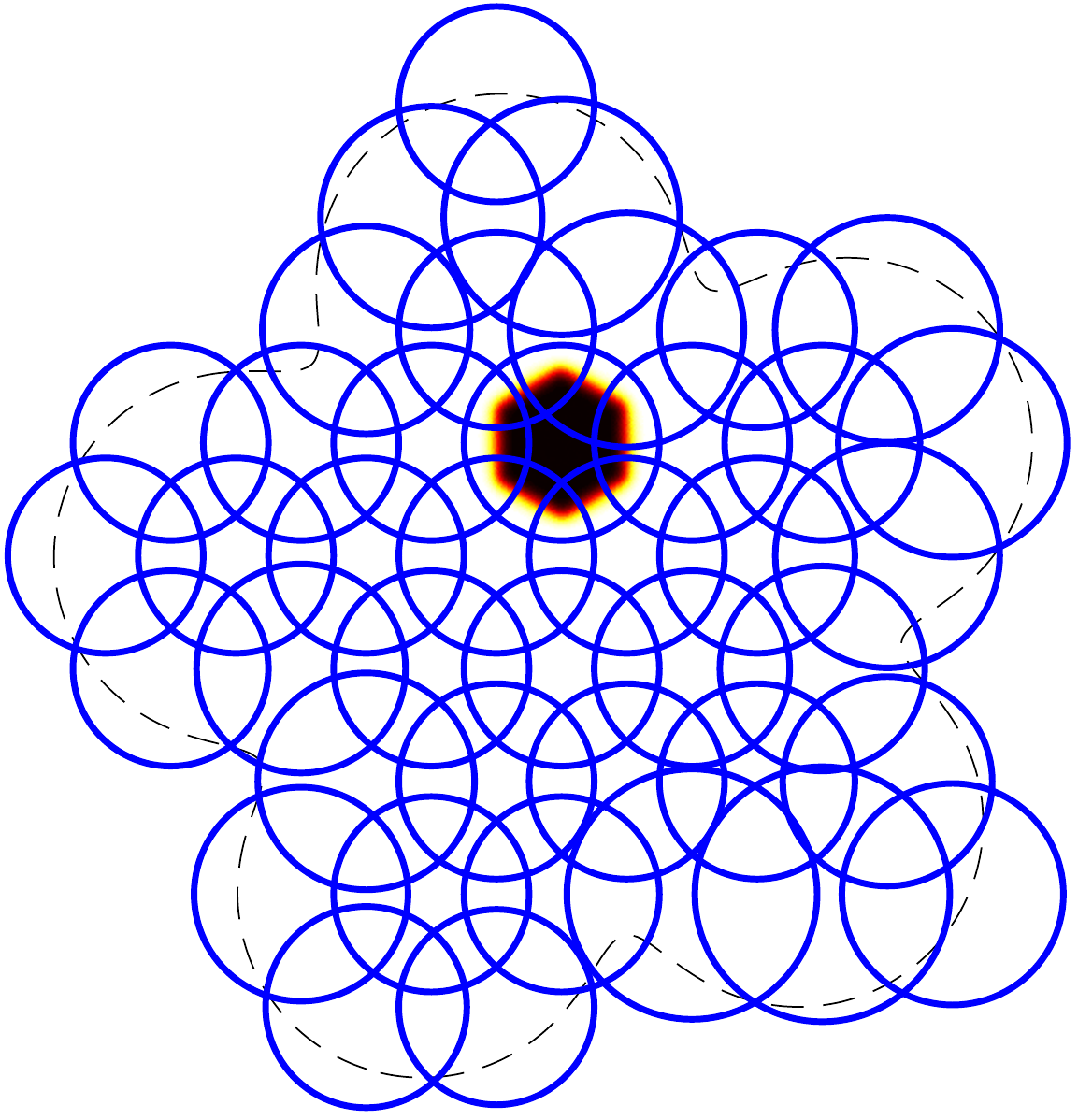} \\
	(a) & (b)
	\end{tabular}
\caption{(a) Illustration of partition of unity patches (outlined in blue lines) for a node set $\nodeset$ (marked with black disks) contained  in a domain $\Omega$ (marked with a dashed line). (b) Illustration of one of the PU weight functions for the patches from part (a), where the color transition from white to yellow to red to black correspond to weight function values from $0$ to $1$.}
\label{fig:RBF_PU}
\end{figure}

\subsection{Description of the method}
A first approach at a vector RBF-PUM may be to construct local vector approximants $\vs_\ell$ for the patches $\Omega_\ell$ that make up the PU using either \eqref{eq:FullDivInterp} for div-free fields or \eqref{eq:FullCurlInterp} for curl-free fields.  These approximants can then be ``blended'' into a global approximant for the underlying field:
\begin{align}
	\vs = \sum_{\ell=1}^M w_\ell \vs_\ell.
	\label{eq:bad_vector_pum}
\end{align}
The issue with this approach is that $\vs$ will not necessarily inherit the div-free or curl-free properties of $\vs_\ell$ because of the multiplication by the weight functions $w_\ell$.  We instead use the local scalar potentials that are recovered from each $\vs_\ell$ and then blend those together.  A div-free or curl-free approximant can then be recovered by applying the appropriate differential operator to the blended potentials.  Since the essential ingredients are very similar for all the kernels treated from Section \ref{sec:review}, for brevity we describe the method only for the div-free case in $\R^2$ and mention any relevant differences as needed.

Let $\nodeset_\ell$ denote the nodes from $\nodeset\subset\R^2$ that belong to patch $\Omega_\ell$, and let $\vs_{\ell}$ denote the div-free RBF interpolant \eqref{eq:FullDivInterp} to the target div-free field $\vu$ over $\nodeset_\ell$.  Our interest is also in the scalar potential for each interpolant given in \eqref{eq:rbf_sf}, which we denote as $\dfpot_{\ell}$.  While we could try to construct a global PU approximant for the scalar potential of the field $\dfpot$ and then apply the operator $\rot$ to the result, we would immediately run into problems since the scalar potentials are only unique up to a constant.  This means that for two patches $\Omega_\ell$ and $\Omega_k$ that overlap, $\dfpot_{\ell}$ and $\dfpot_k$ could be off up to the addition of a constant in the overlap region and thus lead to an inaccurate PU approximant.  To rectify this situation, we need to ``shift'' each $\dfpot_{\ell}$ by a constant $\shift_{\ell}$ such that $\dfpot_{\ell} + \shift_\ell \approx \dfpot_{k} + \shift_k$ if $\Omega_{\ell}$ and $\Omega_k$ overlap.  

To summarize, the main steps of the div-free PUM are as follows:
\begin{enumerate}
\item On each patch $\Omega_\ell$, compute a divergence free interpolant $\vx_{\ell}$ and extract its scalar potential $\dfpot_\ell$ using \eqref{eq:rbf_sf}.
\item Determine constants $\{\shift_\ell\}_{\ell=1}^M$ such that $\tpsi_\ell:=\dfpot_{\ell} + \shift_\ell \approx\dfpot_{k} + \shift_k=:\tpsi_k$ whenever $\Omega_{\ell}\cap\Omega_k \neq \emptyset$.
%\begin{align}
%\underbrace{\dfpot_{\ell} + \shift_\ell}_{\ds =:\tpsi_\ell} \approx \underbrace{\dfpot_{k} + \shift_k}_{\ds =:\tpsi_k}
%\label{eq:adjuste\shift_pot}
%\end{align}
%whenever $\Omega_{\ell}\cap\Omega_k \neq \emptyset$.
\item Blend the shifted potentials with the PU weight functions to obtain a global approximant for the underlying potential:
\begin{equation}
\tpsi(\vx):=\sum_{\ell=1}^M w_\ell(\vx)\tpsi_\ell(\vx).
\label{eq:divfree_sf_pum}
\end{equation}
\item Apply $\rot$  to $\tpsi$ to obtain a global div-free approximant to the underlying field:
  \begin{equation}
    \vst(\vx) :=\sum_{\ell=1}^M \rot \left(w_\ell(\vx)\tpsi_\ell(\vx)\right)=\sum_{\ell=1}^M w_\ell(\vx)\vs_\ell(\vx)+\sum_{\ell=1}^M\tpsi_\ell(\vx)\rot(w_\ell(\mathbf{x})).
  \label{eq:divfree_pum}
  \end{equation}
%\begin{equation}
%\begin{split}
%\vst(\vx)&:=\rot\tpsi(\vx) =\sum_{\ell=1}^M \rot \left(w_\ell(\vx)\tpsi_\ell(\vx)\right)\\
%&=\sum_{\ell=1}^M w_\ell(\vx)\vs_\ell(\vx)+\sum_{\ell=1}^M\tpsi_\ell(\vx)\rot(w_\ell(\mathbf{x})).
%\end{split}
%\label{eq:divfree_pum}
%\end{equation}
\end{enumerate}
Note that the second term in the last equality acts as a correction to the PU approximant formed by blending just the div-free RBF interpolants.  
%As we see from Figure \ref{fig:RBF_PU} (a), the weight functions are constant or nearly constant over a large region, and thus this correction term only makes a noticeable contribution in the transition region of the weight function from 0 to 1.
\begin{figure}[t]
\centering
	\begin{tabular}{cc}
	\includegraphics[width=0.34\textwidth]{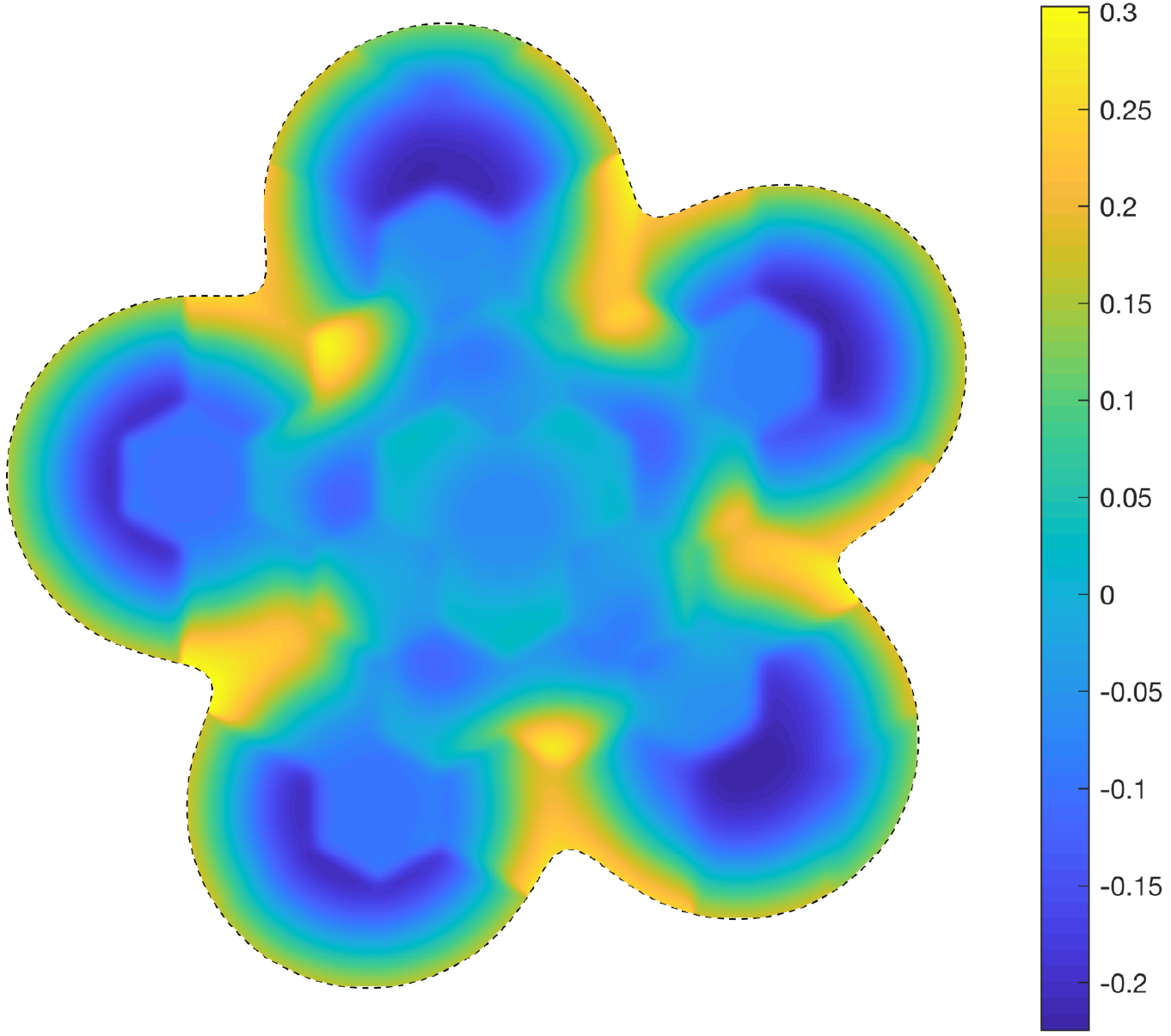} & \includegraphics[width=0.35\textwidth]{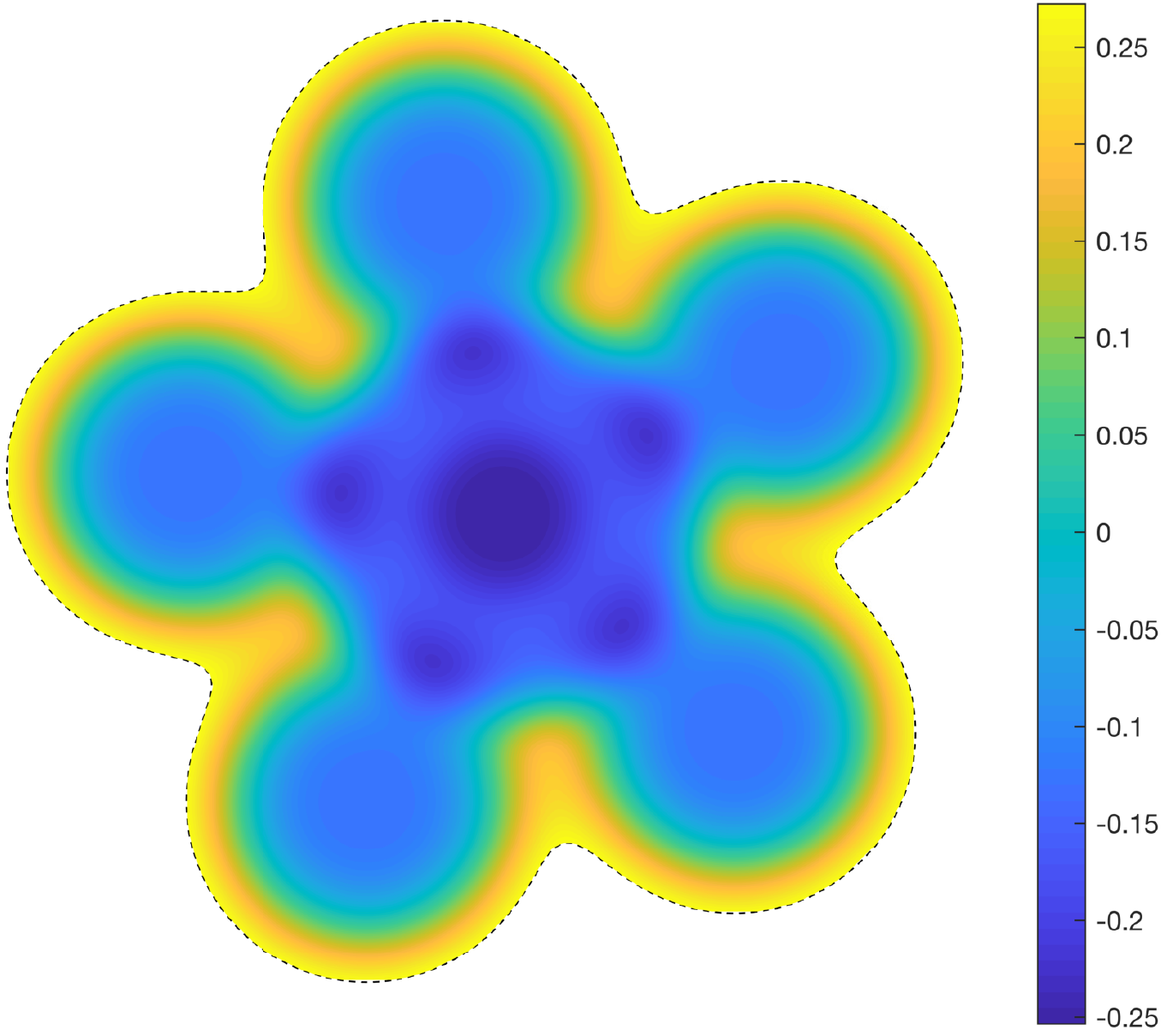} \\
	(a) & (b)
	\end{tabular}
\caption{Div-free RBF partition of unity approximant of the potential from Section \ref{sec:numerics_r2} (a) without the patch potentials shifted ($\dfpot_k$) (b) with the patch potentials shifted ($\tpsi_k$).}
\label{fig:stream_func_pum}
\end{figure}
Figure \ref{fig:stream_func_pum} illustrates the necessity of shifting the patch potentials by way of an example from Section \ref{sec:numerics_r2}.  The figure shows a div-free RBF-PU approximant of a potential when the local patch potentials are not shifted (i.e.,\ using $\dfpot_\ell$ in \eqref{eq:divfree_sf_pum} rather than $\tpsi_\ell$) and when they are shifted.

 \revision{We now turn our attention to a technique for determining the constants $\{\shift_\ell\}_{\ell=1}^M$ for shifting the potential. The idea is to pick a point in the overlap region of each pair of overlapping patches and enforce that the potentials for the each of these patches are equal at this point.  We refer to these points as the ``glue points" since they are where the potentials between neighboring patches are ``glued'' to one another. We have found the following procedure for choosing these points to be effective.  If $\Omega_\ell$ and $\Omega_k$ overlap, then let $\bar{\vx}_{\ell}^k$ denote the center of the overlap region: $\bar{\vx}_{\ell}^k := (\prad_k\vom_\ell + \prad_{\ell}\vom_k)/(\prad_k+\prad_\ell)$, where $\ell < k$ to avoid redundancy; see Figure \ref{fig:Patch_Adjust} for an illustration.}
\begin{figure}[h]
	\centering
	\includegraphics[width=0.5\textwidth]{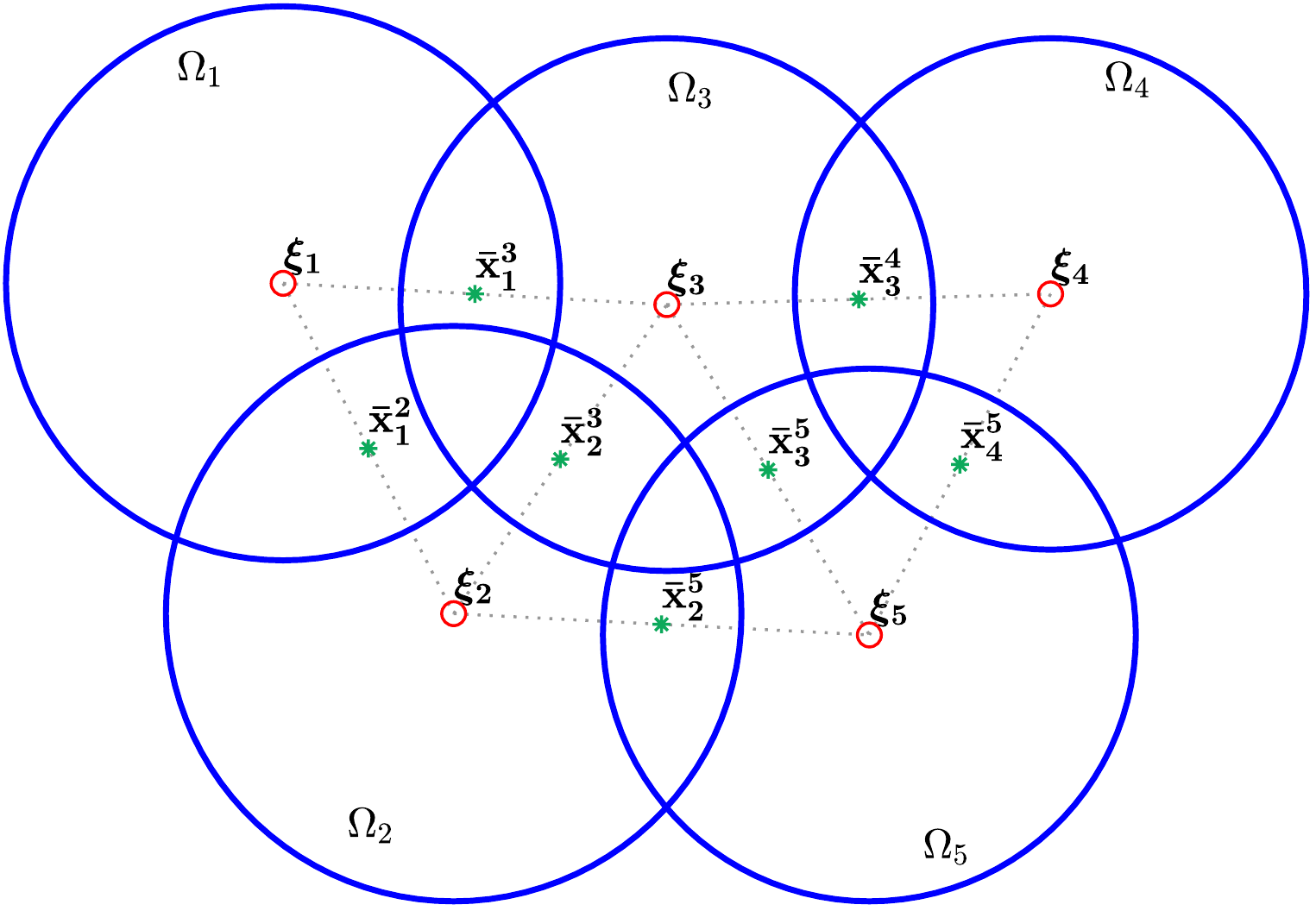}
\caption{Illustration of the glue points for shifting the potentials. The asterisks denote the glue points and the small circles denote the patch centers.}
\label{fig:Patch_Adjust}
\end{figure}
\revision{We denote the collection of all such points by $\bar{X} := \{\bar{\vx}_{\ell}^{k}\,|\, \Omega_\ell\cap \Omega_k\neq \emptyset,\, \ell<k\} = \{\bar{\vx}_i\}_{i=1}^{L}$,
where $L = |\bar{X}|$ and we have reindexed the set so that each $\gluept_i = \gluept_{\ell}^{k}$ for some unique overlapping pair of patches $\Omega_\ell$ and $\Omega_k$.}

On this set we want to impose the conditions
\begin{align*}
  \dfpot_\ell(\bar{\vx}_{\ell}^{k}) + \shift_\ell = \dfpot_k(\bar{\vx}_{\ell}^{k}) + \shift_k
  % \quad\Rightarrow\quad Ad = b,
  %\label{eq:adjustmentsystem}
\end{align*}
for some constants $\shift_\ell$, $\ell = 1,\ldots, M$, which we refer to as the ``potential shifts''.
This can be arranged into a sparse $L$-by-$M$ over-determined linear system
\begin{align} 
P\shift = \rhs
\label{eq:adjustmentsystem}
\end{align} 
with the following properties.  The $L$-by-$M$ matrix $P$ is sparse with two non-zeros per row:  the $i^{\text{th}}$ row, where $\bar{\vx}_{i}$ corresponds to $\bar{\vx}_{\ell}^k$, has a $1$ in the $\ell^{\text{th}}$ column and a $-1$ in the $k^{\text{th}}$ column.  The vector $\shift$ contains the potential shifts, and the vector $\rhs$ is given by $\rhs_{i} =  \dfpot_k(\bar{\vx}_{i}) - \dfpot_\ell(\bar{\vx}_{i}) =  \dfpot_k(\bar{\vx}_{\ell}^k) - \dfpot_\ell(\bar{\vx}_{\ell}^k)$.  The matrix $P$ also has rank $M-1$.  This follows since $P$ is the (oriented) incidence matrix for the graph with vertices being the patch centers $\Omega_\ell$ and edges corresponding to non-empty intersections of patches.  Based on the assumption that $\{\Omega_\ell\}_{\ell=1}^M$ is an overlapping open covering, this graph is connected, so $\text{rank}(P)=M-1$~\cite[Thm.\ 10.5]{GraphTheory}. In the next section we discuss the procedure we use to determine the potential shifts from \eqref{eq:adjustmentsystem}.

\begin{remark}
The procedure described above works exactly the same for curl-free fields in $\R^2$ and $\R^3$ using \eqref{eq:CurlInterp_Rd} for the interpolants and potential fields on each patch.  The procedure also extends to more general surfaces $\M$ for div-free fields (using \eqref{eq:rbf_sf}) and curl-free fields (using \eqref{eq:rbf_vp}).  \revision{However, in this case determining the glue points using the above technique can be more difficult, but for $\M=\sphere^2$, this is easy since the center of the overlap region is trivial to determine}.
\end{remark}

%\begin{remark}
%The procedure above also illustrates why the method does not naturally carry over to div-free fields in $\R^3$.  In this case one obtains vector potentials for each patch and the non-uniqueness is not trivial.  This makes it difficult, if not impossible, to shift the potentials on neighboring patches so that they approximately agree.  \gcomment{Question for Ed: Is the vector potential div-free?}
%\end{remark}

\subsection{Implementation details}\label{sec:implementation}
We now discuss how the patches $\{\Omega_{\ell}\}_{\ell=1}^M$ are chosen as well as how one might compute the potential shifts from the system \eqref{eq:adjustmentsystem}.  In what follows, we assume that the nodes $\nodeset$ are quasiuniformly distributed (i.e.,\ have low discrepancy) in the underlying domain $\Omega$, so that the mesh-norm for $X$,
\begin{align}
h := \sup_{\vy\in \Omega} \min_{\vx\in X} \text{dist}(\vx,\vy),
\label{eq:meshnorm}
\end{align}
satisfies $h = \bigO(1/\sqrt[d]{N})$, where $d$ is the dimension of $\Omega$.  We also assume that there is a signed distance function for the domain to distinguish the interior from the exterior.

\subsubsection{Patch centers}
To determine the patches $\{\Omega_{\ell}\}$ for domains in $\R^2$ and $\R^3$, we use an approach similar to the one described in~\cite{Larsson2017}.  The idea is to start with a regular grid structure of spacing $H$ that covers the domain $\Omega$ of interest and then remove the grid points that are not contained in the domain.  The remaining grid points are chosen as the patch centers $\{\vom_{\ell}\}_{\ell=1}^M$.  Next, an initial radius $\rho$ is chosen proportional to $H$ so the patches $\{\Omega_{\ell}\}_{\ell=1}^M$ form an open cover and there is sufficient overlap between patches (specifics on this are given below).  Finally, for any node in $\nodeset$ that is not contained in one of the patches, the nearest patch center $\vom_j$ is determined and the radius $\prad_j$ for that patch is enlarged to enclose the node.   We perform all range queries on patch centers using a $k$-d tree.

For domains in $\R^2$, we choose the initial grid structure for the patch centers as regular hexagonal lattice of spacing $H$.  Neighboring patches will not overlap if the initial radius is less than or equal to $H/2$.  Therefore, to guarantee overlap, we set the initial radii for the patches to $\rho = (1+\delta)H/2$, where $\delta > 0$. See Figure \ref{fig:RBF_PU} for an illustration of the patches chosen using this algorithm for $\delta=1/2$.  For domains in $\R^3$, we choose the initial grid structure for the patch centers as a regular Cartesian lattice of spacing $H$.  In this case, neighboring patches along the longest diagonal directions will not overlap if the initial radius is less than or equal to $\sqrt{3}H/2$.  To guarantee overlap, we thus set the initial radii for the patches to $\rho = (1+\delta)\sqrt{3}H/2$.

To determine the patches for $\sphere^2$, we use an approach similar to the one described in~\cite{shankar_2018}.  The idea is to use $M$ quasi-uniformly spaced points on $\sphere^2$ for the set of patch centers.  We choose these as near minimum energy (ME) point sets~\cite{HardinSaff:2004}, and use the pre-computed near ones from~\cite{SpherePts}.  For a set with $M$ points, the average spacing $H$ between the points can be estimated as $H\approx \sqrt{4\pi/M}$.  We select a value of $H$ and then determine $M$ as $M=\lceil 4\pi/H^2\rceil$.  Since the ME points are typically arranged in hexagonal patterns (with a few exceptions~\cite{HardinSaff:2004}), we choose the radius for each patch as $\rho_{\ell} = (1 + \delta)H/2$, where the parameter $\delta$ again determines the overlap.  

To keep the overall cost under control, the initial radii of the patches $H$ should decrease as $N$ increases.  The rate at which $H$ should decrease can be determined as follows.  Assuming that the patches that intersect the boundary have similar radii to the interior patches, and using the assumption that $\nodeset$ is quasiuniform, a simple volume argument gives that number of nodes in each patch satisfies $n = \bigO(\rho^d N)=\bigO(H^d N)$, where $d$ is the dimension of $\Omega$.  So, to keep the work roughly constant per patch, we need $H=\bigO(1/N^{1/d})$.  In our implementation of the vector PUM, we choose
\begin{align}
H = q \left(A/N\right)^{1/d},
\label{eq:H_heuristic}
\end{align}
where $A$ is related to the area/volume of $\Omega$, and $q$ is a parameter that controls the average number of nodes per patch.  Note that from the above analysis, the computational cost increases as the overlap parameter increases and as $q$ increases.  Based on the assumptions on $X$ and the patches, choosing $H$ according to \eqref{eq:H_heuristic} results in a computational cost of $\bigO(N)$ for constructing the vector PUM approximants, and $\bigO(N\log N)$ for the range queries involved for determining the patch structure.  However, in practice, the cost is dominated by the former part of the method.

\subsubsection{Potential shifts}
Since $\text{rank}(P)=M-1$ and its nullspace consists of constant vectors, we first set one of the shifts $\shift_j$ to zero, for some $1\leq j\leq M$, and then compute the remaining shifts using the least squares solution of \eqref{eq:adjustmentsystem}.  For this problem we can form the normal equations directly since the matrix $P^TP$ is just a graph Laplacian (recall $P$ is an oriented incidence matrix).  We have found that the accuracy of the reconstructed field \eqref{eq:divfree_pum} can often be improved if a weighted least squares approach is used.  In this case, we use a diagonal weight matrix $W$ with entries that depend on the distance between the glue points and the patch centers.  Specifically, we set $r_i$ as the closer of the two distances between the $i^{\rm th}$ glue point $\bar{\vx}_i$ and the centers of the two patches it was formed from, and then set 
\begin{align}
W_{ii} = \exp\left(-\gamma\left(1-\frac{r_i}{r_{\rm min}}\right)^2\right),
\label{eq:wght_matrix}
\end{align}
where $r_{\rm min} = \min_{j} r_j$ and $\gamma > 0$.  The normal equations in this case now look like a weighted graph Laplacian.

%%%%%%%%%%%%%%%%%%%%%%%%%%%%%%%%%%
% Error estimates
%%%%%%%%%%%%%%%%%%%%%%%%%%%%%%%%%%%%%%%%%%%%%%%%%%%%%%%%%%%%%%%%%%%
\section{Error Estimates}\label{mainerrorsection}
%%%%%%%%%%%%%%%%%%%%%%%%%%%%%%%%%%%%%%%%%%%%%%%%%%%%%%%%%%%%%%%%%%%%
%Let $\Omega\subset \R^2$, $X\subset \Omega$ be a finite collection of nodes, and let $\{\patch_\ell,w_\ell\}_{\ell=1}^n$ be a partition of unity for $\Omega$.
%As before, we denote the nodes on a given patch by $\nodeset_\ell = \nodeset \cup \patch_\ell$.
The error bounds will be expressed in terms of local mesh norms $h_\ell$, 
%which for each patch $\patch_\ell$ is the radius of the largest open ball with center in $\patch_\ell$ that does not intersect the nodes $\nodeset_\ell$ on that patch, given via %$h_\ell := \sup_{y\in \patch_\ell} \min_{\node\in \nodeset_\ell} \text{dist}(\node,y)$.
%\[h_\ell := \sup_{\vy\in \patch_\ell} \min_{\node\in \nodeset_\ell} \text{dist}(\node,\vy).\]
which are given by \eqref{eq:meshnorm}, with $\Omega=\Omega_{\ell}$ and $X=X_\ell$.
Error rates for RBF interpolation, including divergence-free (curl-free) RBF approximation, both in flat space and on the sphere, have been known for some time. Many of these estimates are valid for target functions within the \emph{native space}, which we denote by $\mathcal{N}(\Omega)$, of the RBF used - which for infinitely smooth RBFs are subspaces of analytic functions and for kernels of finite smoothness are essentially Sobolev spaces (with norms equivalent to Sobolev norms on bounded subsets)\footnote{See \cite[Ch. 10]{WendlandScatteredData} for native spaces of scalar valued functions, and see \cite{Fuselier2008:StabilityNative,FuselierEtAl2009:SphereDivFree} for the vector cases on $\mathbb{R}^d$ and the sphere.}. For the RBF kernels considered here, there is a continuous embedding from the native space of the matrix kernel into a Sobolev space of order $\tau > d/2$. In this situation we get the estimate below. \revision{In what follows, we let $\vH^\tau(\Omega_\ell)$ denote the space of vector fields with each coordinate function in the Sobolev space $H^\tau(\Omega)$ with smoothness $\tau$.

\begin{proposition}\label{prop:localerror}
Suppose that $\vu\in \mathcal{N}(\patch)$ and that each $\patch_\ell\subset \patch$ satisfies an interior cone condition with radius $R_\ell$ and angle $\theta$ independent of $\ell$. Suppose also that there is a continous embedding of the native space into $\vH^\tau(\Omega)$. Then there are constants $Q:=Q(\theta,\tau)$ and $C:=C(\theta,\tau,d)$ such that if $h_\ell<Q R_\ell$, then
  %Suppose also that the native space of $\phi$ is continuously embedded in a Sobolev space of order $\tau>d/2$. Then there is a constant $D:=D(\tau,\theta) > 0$ such that if $\text{diam}(\Omega_\ell)\geq D h_\ell$.
\[\|\vu - \vs_{\ell}\|_{L_\infty(\patch_\ell)} \leq \mathcal{E}(h_\ell)\|\vu\|_{\mathcal{N}(\patch_\ell)},\]
where $\mathcal{E}(h) = Ch^{\tau - d/2}$. %If the kernel is infinitely smooth, then $\mathcal{E}(h)\rightarrow 0$ faster than any fixed algebraic rate. 
\end{proposition}

\begin{proof}
Estimates like these have been worked out for div/curl-free RBFs on subsets of $\R^d$ and on $\sphere^2$~\cite{Fuselier2008:VectorError,FuselierEtAl2009:SphereDivFree, FuselierWright2009:SphereDecomp}. However, in the papers referenced the domain was fixed and the dependence of the constants on the cone condition radius was not emphasized, so we should briefly review the arguments here.
  %- but here the size of the domain $\Omega_\ell$ should scale with $h_\ell$, so we should briefly review why the constant $C$ does not depend on the size of the domain $\patch_\ell$.

First, note that the function $\vu - \vs_{\ell}$ will be zero on $\nodeset_\ell$. On domains satifying an interior cone condition, in the Euclidean case and on surfaces, we may therefore employ a ``zeros lemma'' in each coordinate function. These give constants $Q:=Q(\theta,\tau)$ and $C:=C(\theta,\tau,d)$ such that if $h_\ell<Q R_\ell$, then
\[\|\vu - \vs_{\ell}\|_{L_{\infty}(\Omega_\ell)} \leq C h_{\ell}^{\tau - d/2}\|\vu - \vs_{\ell}\|_{\vH^\tau(\Omega_\ell)}.\]
See for example \cite[Theorem 11.32]{WendlandScatteredData} and\cite[Theorems A.4 and A.11]{HangelbroekEtAl:2012PolyharmI}). 

Next, since $\vu\in \mathcal{N}(\Omega)$, then $\vu\in \mathcal{N}(\Omega_\ell)$ and there is an isometric extension $E:\mathcal{N}(\Omega_\ell) \rightarrow \mathcal{N}(\Omega)$ such that $\|E\vu\|_{\mathcal{N}(\Omega)} = \|\vu\|_{\mathcal{N}(\Omega_\ell)}$ (see \cite[Theorem 10.46,10.47]{WendlandScatteredData}\footnote{The theorems referenced are given in the Euclidean scalar-valued context, but the arguments are general enouch to apply to matrix valued positive definite kernels on any set.}). With this, since $\mathcal{N}(\Omega)$ is \revision{continuously} embedded in $\vH^\tau(\Omega)$ for some $\tau>d/2$, we get
\begin{eqnarray*}
\|\vu - \vs_{\ell}\|_{\vH^\tau(\Omega_\ell)}&=& \|E\vu - \vs_{E\vu,\ell}\|_{\vH^\tau(\Omega_\ell)} \leq \|E\vu - \vs_{E\vu,\ell}\|_{\vH^\tau(\Omega)} \leq C \|E\vu - \vs_{E\vu,\ell}\|_{\mathcal{N}(\Omega)},
\end{eqnarray*}
where we write $\vs_{E\vu,\ell}= \vs_\ell$ to emphasize that the interpolant on $\nodeset_\ell$ of the extension is also $\vs_\ell$. Note that the constant here may depend on $\Omega$, but not on $\Omega_\ell$. Finally, it is well-known that the interpolation error is always orthogonal to the kernel interpolant in the native space, which implies the bound
\[\|E\vu - \vs_{E\vu,\ell}\|_{\mathcal{N}(\Omega)} \leq \|E\vu\|_{\mathcal{N}(\Omega)} = \|\vu\|_{\mathcal{N}(\Omega_\ell)},\]
where the last equality follows because $E$ is an isometry. This completes the proof.
\end{proof}
Thus it is possible to acheive high order convergence with patch sizes that are proportional to the mesh norm. In what follows we assume that the patch radii and local mesh norms are such that Proposition \ref{prop:localerror} is satisfied.

%For example, in the simple case where the points are quasi-uniform so that $h_\ell\sim h$ and the patches are simply balls in Euclidean space, giving $r_\ell = \text{diam}(\Omega_\ell)/2$, having patch sizes satisfying $\text{diam}(\Omega_\ell) \sim (2/Q)h$ ensures high order convergence. 
%In practice, we choose the patch sizes large enough to guarantee approximation by polynomials of a particular degree, and this seem to be sufficient. ---   
}

%F = sin(pi/3)/(4*(1 + sin(pi/3)))
%v = 2*asin(F)
%Q = sin(pi/3)*sin(v)/(8*(1+sin(pi/3)*(1 + sin(v))))
%Q
%Q4 = Q/4^2
%ans =  0.00075493

%%%%%%%%%%%%%%%%%%%%%%%%%%%%%%%%%%%%%%%%%%%%%%%%%%%%%%%%%%%%%%%%%%%%%%%%%
In addition to the estimate above, our arguments that follow will also rely on the Mean Value Theorem, which for a scalar function $\dfpot\revision{:\R^d\rightarrow \R}$ and $\vx,\vy\in \R^d$ we express as
\begin{equation*}%\label{MVTflat}
  |\dfpot(\vx) - \dfpot(\vy)| \leq \left|\nabla(\dfpot)|_{\vx^*}\right|\,\dist(\vx,\vy),
\end{equation*}
where $\vx^*$ is on the line segment between $\vx$ and $\vy$. Here we use the notation $|\cdot|$ to denote the Euclidean length when the argument is a vector. To derive a similar estimate on surfaces, let $\vx,\vy\in \M$ and let $\gamma:[0,\sdist(\vx,\vy)]\rightarrow \M$ denote a shortest path in $\M$ connecting $\vx$ and $\vy$ with $\gamma(0)=\vx$, $\gamma(\sdist(\vx,\vy))=\vy$, parameterized by arclength. This implies that $\gamma'$ is tangent to $\M$ and $|\gamma'|=1$. Applying the single variable Mean Value Theorem to the real-valued function $\psi\circ\gamma$ implies that
\[
|\dfpot(\vx) - \dfpot(\vy)| \leq |\nabla\psi\cdot\gamma'|_{t^*}|\sdist(\vx,\vy),\]
where $t^*\in [0,\sdist(\vx,\vy)]$. Since $\gamma'$ is tangent to $\M$ and has length $1$, we get $|\nabla\psi\cdot\gamma'| = |\sgrad\dfpot\cdot\gamma'|\leq |\sgrad\dfpot|$. Combining the above with the fact that $|\surfacegrad(\dfpot)| = |\surfacecurl(\dfpot)|$ gives us the following
\begin{equation}\label{MVTsurface}
|\dfpot(\vx) - \dfpot(\vy)|\leq  \left|\surfacegrad(\dfpot)|_{\vx^*}\right|\,\sdist(x,y) = \left|\surfacecurl(\dfpot)|_{\vx^*}\right|\,\sdist(\vx,\vy),
\end{equation}
where $\vx^*\in \M$. 

%\gcomment{In Section 3.1 we use $K$ for this value}.\ecomment{$K$ (max number of patches that contain the same point) and $m$ (max number of patches that intersect a given patch) are related but measuring two different things}

Before proceeding we summarize some of the important assumptions on the partition of unity. Recall that each $\vx\in\Omega$ is covered by only a small number of patches (say at most $K$ patches). We also assume that the number of patches that intersect a given patch is uniformly bounded by some constant $m$. Additionally, we suppose that there are roughly the same number of nodes in each patch, and that the node distribution in each patch is quasi-uniform. This leads to an estimate of the form $c h_\ell \leq \text{diam}(\patch_\ell)\leq C h_\ell$ for some constants $c,C$ independent of $\ell$. Lastly, we assume that the partition is ``1-stable'' (see \cite{WendlandScatteredData}[Def. 15.16]), meaning that first order derivatives of the weight functions satisfy a bound of the form $|\nabla w_\ell|\leq C(\text{diam}(\patch_\ell))^{-1}$, where $C$ is some constant independent of $\ell$. This with the quasi-uniformity supposition gives the bound $|\nabla{w}_\ell|=|\surfacecurl w_\ell|\leq Ch_\ell^{-1}$ for some $C$ independent of $\ell$. 

Now we give an estimate for the pointwise error of the divergence-free approximant in a two dimensional domain. Note that the bound is local in the sense that it comprised of a local interpolation error plus an expression involving the residuals $r_\ell^k:=\localadjdfpot_{\ell}(\gluept_\ell^k) - \localadjdfpot_k(\gluept_\ell^k)$ from adjusting neighboring potential functions. 
\begin{theorem}\label{theorem:errorestimates}
\revision{Suppose that the conditions in Proposition 4.1 are satisfied.} Given a div-free vector field $\vu = \surfacecurl(\dfpot)\in \mathcal{N}(\Omega)$, let $\blendedpotential$ and $\dfpuapprox = \surfacecurl(\blendedpotential)$ denote the PUM approximants from \eqref{eq:divfree_sf_pum} and \eqref{eq:divfree_pum}. Then the error at $\vx \in \Omega$ satisfies 
\begin{eqnarray}%\label{eq:errorestimates}
  \left|\surfacegrad(\blendedpotential - \dfpot)(\vx)\right| &=& \left|\surfacecurl(\blendedpotential - \dfpot)(\vx)\right| = \left|\vu(\vx) - \dfpuapprox(\vx)\right|\nonumber\\
  &\leq &m C\max_{\ell\,|\,\vx\in \patch_\ell} \left(\mathcal{E}(h_\ell)\|\vu\|_{\mathcal{N}(\patch_\ell)}\right) + C\sum_{\ell | \vx\in \patch_\ell,\, \ell\neq k}h_\ell^{-1}|r_\ell^k|,\label{eq:errorestimates}
\end{eqnarray}
%C \left[\max_{\ell\,|\,x\in \patch_\ell} \mathcal{E}(h_\ell)\|\vu\|_{\mathcal{N}(\patch_\ell)} + \max_{\ell\,|\,x\in \patch_\ell}(h_\ell^{-1})\sqrt{\sum_j h_j^2\mathcal{E}(h_j)^2\|\vu\|^2_{\mathcal{N}(\patch_j)}}\right],\]
where $k$ is any index such that $\vx\in \patch_k$.
\end{theorem}

\begin{proof}
The first equality follows from the fact that $\surfacegrad f$ and $\surfacecurl f$ have the same magnitude. Next, note that
\begin{equation}\label{interpolantandcorrection}
\dfpuapprox = \sum_\ell w_\ell \vs_{\ell} + \sum_\ell\surfacecurl(w_\ell)\tpsi_\ell.
\end{equation}
The first term is a weighted average of RBF interpolants to $\vu$ and the weight functions sum to 1, so we have
\begin{eqnarray*}
\left|\vu(\vx) - \sum_\ell w_\ell(\vx) \vs_{\ell}(\vx)\right| & = & \left|\sum_\ell w_\ell(\vx)\vu(\vx) - \sum_\ell w_\ell(\vx)\vs_{\ell}(\vx)\right|\leq \sum_\ell w_\ell(\vx)|\vu(\vx) - \vs_{\ell}(\vx)|\\
&\leq &\sum_\ell w_\ell(\vx)C\mathcal{E}(h_\ell)\|\vu\|_{\mathcal{N}(\patch_\ell)}= C\max_{\ell\,|\,\vx\in \patch_\ell} \mathcal{E}(h_\ell)\|\vu\|_{\mathcal{N}(\patch_\ell)}. 
\end{eqnarray*}

To complete the proof we need to bound the second term in \eqref{interpolantandcorrection}. Given $\vx\in \Omega$, fix a $k$ such that $\vx\in \patch_k$. Since $\sum \surfacecurl(w_\ell) = 0$ and $w_\ell(\vx) = 0$ for $\vx\notin \patch_\ell$ we get
\[  \sum_\ell\surfacecurl(w_\ell)\tpsi_\ell(\vx) =  \sum_{\ell\,|\,\vx\in \patch_\ell} \surfacecurl(w_\ell)\left(\tpsi_\ell(\vx) - \tpsi_k(\vx)\right).\]
This and our assumptions on the weight functions give us the estimate
\begin{equation}\label{correctionbound}
  \left|\sum_\ell\surfacecurl(w_\ell)\tpsi_\ell(\vx)\right| \leq \sum_{\ell\,|\,\vx\in \patch_\ell} C h_\ell^{-1}\left |\tpsi_\ell(\vx) - \tpsi_k(\vx)\right|.% \leq m\,C \max_{\ell\,|\,x\in \patch_\ell} h_\ell^{-1}\left |\tpsi_\ell(x) - \tpsi_k(x)\right|.
\end{equation}
If $\ell = k$, the corresponding term in the sum is zero. If $\ell\neq k$, we let $g:=\localadjdfpot_\ell - \localadjdfpot_k$ and $\gluept_\ell^k$ be the adjustment point for $\patch_\ell$ and $\patch_k$, we can rewrite
\begin{equation*}
\tpsi_\ell(\vx) - \tpsi_k(\vx) = g(\vx)- g(\gluept_\ell^k) + g(\gluept_\ell^k) = g(\vx)- g(\gluept_\ell^k) + r_\ell^k.
\end{equation*}
To bound $g(\vx)- g(\gluept_\ell^k)$, we use \eqref{MVTsurface} and the fact that $\surfacecurl(g) = \vs_{\ell} - \vs_{k}$ to get
\begin{eqnarray*}
|g(\vx)- g(\gluept_\ell^k)| &\leq& \|\surfacecurl(g)\|_{L_\infty(\patch_k\cap \patch_\ell)}\text{dist}(\vx,\gluept_\ell^k) \leq \|\surfacecurl(g)\|_{L_\infty(\patch_k\cap \patch_\ell)}h_{\ell} \nonumber \\
  & \leq & h_{\ell} \left(\| \vs_{\ell} - \vu \|_{L_\infty(\patch_k\cap \patch_\ell)} + \| \vu - \vs_{k}\|_{L_\infty(\patch_k\cap \patch_\ell)})\right)\nonumber \\
  & \leq & C h_{\ell} \left(\mathcal{E}(h_\ell)\|\vu\|_{\mathcal{N}(\patch_\ell)} + \mathcal{E}(h_k)\|\vu\|_{\mathcal{N}(\patch_k)}\right), \nonumber 
\end{eqnarray*}
which when applied to \eqref{correctionbound} gives 
\begin{eqnarray*}
  \left|\sum_\ell\surfacecurl(w_\ell)\tpsi_\ell(\vx)\right| &\leq& \sum_{\ell | \vx\in \patch_\ell,\,\ell \neq k} C\left(\mathcal{E}(h_\ell)\|\vu\|_{\mathcal{N}(\patch_\ell)} + \mathcal{E}(h_k)\|\vu\|_{\mathcal{N}(\patch_k)}\right) + C h_\ell^{-1}|r_\ell^k|\\
& \leq & m C\max_{\ell\,|\,\vx\in \patch_\ell} \mathcal{E}(h_\ell)\|\vu\|_{\mathcal{N}(\patch_\ell)} + C\sum_{\ell | x\in \patch_\ell,\, \ell\neq k}h_\ell^{-1}|r_\ell^k|.
\end{eqnarray*}
The result follows.
\end{proof}

Note that very similar arguments follow through also for curl-free vector fields on surfaces, i.e. an estimate identical to \eqref{eq:errorestimates} holds for the curl-free case. The proof also carries directly over to $\R^d$ - namely if $\vu = \nabla \cfpot$, and $\dfpuapprox = \nabla\widetilde{\cfpot}$ denotes the curl-free RBF-PUM approximant, one has an estimate of the form
\begin{equation*}%\label{eq:errorestimatesRd}
\left|\nabla(\widetilde{\cfpot} - \cfpot)(\vx)\right| = \left|\vu(\vx) - \dfpuapprox(\vx)\right| \leq m C\max_{\ell\,|\,\vx\in \patch_\ell} \left(\mathcal{E}(h_\ell)\|\vu\|_{\mathcal{N}(\patch_\ell)}\right) + C\sum_{\ell | \vx\in \patch_\ell,\, \ell\neq k}h_\ell^{-1}|r_\ell^k|.
\end{equation*}

Now we discuss the residual in shifting the local potentials. We begin by showing that good constants for the shifts exist.
\begin{proposition}\label{prop:potentialshifts}
Let $\vs_\ell = \rot \dfpot_{\ell}$ be the local RBF interpolant on $\nodeset_\ell\subset \Omega_\ell$ and let $\glueptset_\ell = \glueptset\cap \patch_\ell$ be the collection of glue points on $\patch_\ell$. Given any $v$ such that $\vu = \surfacecurl(v)$, the constant
\[\shift^*_\ell := \displaystyle \frac{1}{|\glueptset_\ell|}\sum_{\vy\in \glueptset_\ell}v(\vy) - \dfpot_\ell(\vy)\]
gives
\[\|\dfpot_\ell + \shift^*_\ell - v\|_{L_\infty(\patch_\ell)} \leq C h_\ell\mathcal{E}(h_\ell) \|\vu\|_{\mathcal{N}(\patch_\ell)}.\]
\end{proposition}

\begin{proof}
Let $\vx\in \patch_\ell$. First we apply the triangle inequality and the Mean Value Theorem to obtain
\begin{eqnarray*}
  \left|\dfpot_\ell(\vx) + \shift_\ell^* - v(\vx)\right| &\leq& \frac{1}{|\glueptset_\ell|}\sum_{\vy\in \glueptset_\ell}\left|\dfpot_\ell(\vx) - v(\vx) - (\dfpot_\ell(\vy) - v(\vy))\right| \\
%  &\leq&  \frac{1}{|\glueptset_\ell|}\sum_{y\in \glueptset_\ell}\left\|\surfacecurl(\dfpot_\ell - v)\right\|_{L_\infty(\patch_\ell)}\text{dist}(x,y) \\
  &\leq& \frac{1}{|\glueptset_\ell|}\sum_{\vy\in \glueptset_\ell}\left\|\vs_{j} - \vu \right\|_{L_\infty(\patch_\ell)}\text{dist}(\vx,\vy). 
\end{eqnarray*}
Next, an application of Proposition \ref{prop:localerror} and the fact that $\text{diam}(\patch_\ell)\leq C h_\ell$ finishes the proof. 
\end{proof}

Letting $\res^*:= P\shift^* - \rhs$, i.e., the residual in the system \eqref{eq:adjustmentsystem} using the shifts given in the above proposition, with a triangle inequality and using the fact that $h_k\sim h_\ell$ for neighboring patches, we get
\begin{equation}\label{adjustmentresidualbound}
%\begin{aligned}
  (\res^*)_\ell^k \leq C h_\ell\mathcal{E}(h_\ell) \|\vu\|_{\mathcal{N}(\patch_\ell)} + C h_\ell \mathcal{E}(h_k) \|\vu\|_{\mathcal{N}(\patch_k)}.
%\end{aligned}
\end{equation}
Applying this to the residual term from \eqref{eq:errorestimates} becomes:
\begin{eqnarray}
\sum_{\ell | \vx\in \patch_\ell,\, \ell\neq k}h_\ell^{-1}(\res^*)_\ell^k &\leq & m C \max_{\ell\,|\,\vx\in \patch_\ell} \mathcal{E}(h_\ell)\|\vu\|_{\mathcal{N}(\patch_\ell)}
\end{eqnarray}
Thus if the shifts are chosen appropriately the method can achieve the same approximation order as that of local interpolation. However, we compute the shifts according to the overdetermined \eqref{eq:adjustmentsystem}. The residual from that system satisfies the following.
\begin{proposition}\label{residualbound}
Let $\shift$ be the least squares solution to \eqref{eq:adjustmentsystem}. The residual $\res:= P\shift - \rhs$ satisfies the bound \[|\res|^2 \leq  m\,C\sum_\ell h_\ell^2\mathcal{E}(h_\ell)^2\|\vu\|^2_{\mathcal{N}(\patch_\ell)}.\]
\end{proposition}
\begin{proof}
Choose any scalar potential $v$ such that $\vu = \surfacecurl(v)$, and let $\shift^*$ be the vector whose $\ell^{th}$ element is $\shift_\ell^*$ as defined in Proposition \ref{prop:potentialshifts}. Then we have $|\res| \leq |\res^*|$. Next, we square the left-most inequality in \eqref{adjustmentresidualbound} and estimate further to get
\begin{equation}
((\res^*)_{\ell}^k)^2\leq C \left(\mathcal{E}(h_\ell)^2h_\ell^2\|\vu\|_{\mathcal{N}(\patch_\ell)}^2 +  \mathcal{E}(h_k)^2h_k^2\|\vu\|_{\mathcal{N}(\patch_k)}^2\right).
\end{equation}
Now sum the estimate over all glue points, and note that each $\patch_\ell$ (and $\patch_k$) will appear in the sum at most $m$ times (the maximum number of patches that intersect any given patch). This gives the result.
\end{proof}

In an attempt to bound the error solely in terms of the point distribution and target function, let us look at an application of this estimate to the residual term from \eqref{eq:errorestimates}. For simplicity, assume that all $h_\ell\sim h$ for all $h_\ell$. Since there are at most $m$ terms in the sum, a Cauchy-Schwarz inequality gives
\begin{eqnarray*}
\sum_{\ell | \vx\in \patch_\ell,\, \ell\neq k}h_\ell^{-1}|r_\ell^k| \leq h^{-1}\sqrt{m}|\res| \leq C m \mathcal{E}(h)\sqrt{\sum_\ell\|\vu\|_{\mathcal{N}(\Omega_\ell)}^2}.
\end{eqnarray*}
Due to the sum over all patches, this bound may or may not match the expected error rates. \revision{It is reasonable to guess that this sum is equivalent to $\|\vu\|^2_{\mathcal{N}(\Omega)}$. Numerical experiments for scalar RBF interpolants, not presented here, suggest that such a sum may be uniformly bounded in the case of a thin plate spline, and may grow very slowly for Mat\'ern kernels. We leave exploring a tight bound for this term as an open question. A} very rough estimate of the sum would introduce a factor of $\sqrt{M}$, where $M$ is the number of patches. In the quasi-uniform case, a volume argument gives $\sqrt{M}\sim h^{-d/2}$. Thus a worst-case scenario is that the method converges according to $\mathcal{E}(h)h^{-d/2}$. However, numerical experiments suggest that the errors decay according to $\mathcal{E}(h)$ (see for example Section \ref{numerics_sphere}) and do not seem to depend on the number of patches - which suggests that the estimate $\mathcal{E}(h)h^{-d/2}$ is pessimistic. %We conjecture that the correct error rate is indeed $\mathcal{E}(h)$.

%%%%%%%%%%%%%%%%%%%%%%%%%%%%%%%%%%%%%%%%%%%%%%%%%%%%%%%%%%%%%%%%%%%%%%%%%%%%
%However, consider the case where $\mathcal{N}(U_\ell) = H^\tau(U_\ell)$, with the standard Sobolev norm. Expressing the norm in each patch as an integral, one can use the fact that each point is contained in as most $m$ patches to arrive at $\sum_\ell\|\vu\|^2_{H^\tau(U_\ell)}\sim \|\vu\|_{H^\tau(\Omega)}$, and in this case we see the potentially troublesome term is not a problem. In general we do not know how to prove this.

%%%%%%%%%%%%%%%%%%%%%%%%%%%%%%%%%%

%%%%%%%%%%%%%%%%%%%%%%%%%%%%%%%%%%
\section{Numerical experiments}\label{numerics}
%%%%%%%%%%%%%%%%%%%%%%%%%%%%%%%%%%
In this section, we numerically study the vector RBF-PUM for three different test problems: a div-free field in a star-shaped domain in $\R^2$, a div-free field on $\sphere^2$, and a curl-free field in the unit ball in $\R^3$.  For each of these cases, we numerically test the convergence rates of the method and compare them to the estimates from Section \ref{mainerrorsection}.  The point sets we use in the experiments are all quasiuniform, so rather than compute the mesh-norm $h$ and use this to measure convergence rates, we simply use $h\sim N^{-1/d}$.  

To illustrate the different convergence rates that are possible, we use the inverse multiquadric (IMQ) kernel $ \phi(r) = 1/\sqrt{1 + (\ep r)^2}$ and the Mat\'ern kernel $\phi(r) = e^{-\ep r}\left(1 + (\ep r) + \frac{3}{7}(\ep r)^2 + \frac{2}{21}(\ep r)^3 + \frac{1}{105}(\ep r)^4\right)$.\   The latter kernel is piecewise smooth and the local error from Proposition \ref{prop:localerror}, in terms of $N$, is given by $\mathcal{E}(N) = (\sqrt{N})^{-3.5}$ for $d=2$ (see~\cite{FuselierWright2009:SphereDecomp} for more details). The IMQ kernel is analytic and therefore the local error decreases faster than any algebraic rate.  For scalar interpolation with the IMQ, the local error estimate is $\mathcal{E}(N) = e^{-C\log(N)N^{1/2d}}$~\cite{RiegerZwicknagl_samplinginequalities}, where $C>0$ is a constant.  We demonstrate that this also appears to be the correct rate for the vector case.  While the error estimates are in terms the $\infty$-norm, we also include results on the $2$-norm for comparison purposes.  
%We use the IMQ kernel $\phi(r) = 1/\sqrt{1 + (\ep r)^2}$, where $\ep > 0$ is the shape parameter for all the tests.  The fields we use are all smooth, so for this kernel, we expect the local error estimate $\mathcal{E}(h)$ from Proposition \ref{prop:localerror} to be $\mathcal{E}(h) = e^{C\log(h)/\sqrt{h}}$~\cite{RiegerZwicknagl_samplinginequalities}.  \ecomment{I'm working on a brief explanation of why we expect this since the theorem doesn't say it explicitly. According to Cor 5.1 in Barbara and Christian's paper, shouldn't we see $e^{-C/h}$? Does it make sense that we're not given that the number of local nodes is roughly staying fixed? Did you see $e^{-C/h}$ in the case of fixed patches?} 
Since we are interested in demonstrating the convergence rates from the theory, we fix the shape parameter $\ep$ in all the tests, as using different $\ep$ on a per patch level will lead to different constants in the estimates.  The values were selected so that conditioning of the linear systems \eqref{eq:divfree_linsys} (or \eqref{eq:curlfree_linsys}) is not an issue.  Choosing variable shape parameters in scalar RBF-PUM is explored in~\cite{cavoretto_2016} and may be adapted to the current method, but we leave that to a separate study.  For brevity we report results for one kernel per example, with the IMQ kernel used for the first and third test and the Mat\'ern used for the second.  However, we note that the estimated convergence rates for each kernel were consistent with the theory across all tests. Finally, we set the weighted least squares parameter in \eqref{eq:wght_matrix} to $\eta=4$.  This value produced good results over all the numerical experiments performed.  

\revision{All results were obtained from a MATLAB implementation of the vector RBF-PUM method executed on a MacBook Pro with 2.4 GHz 8-Core Intel Core i9 processor and 32 GB RAM.  No explicit parallelization was implemented.}

%%%%%%%%%%%%%%%%%%%%%%%%%%%%%%%%%%%%%%%%%%%%%%%%%%%%%%%%%%%%%%%%%%%%
\subsection{Div-free field on $\mathbb{R}^2$}\label{sec:numerics_r2}
%%%%%%%%%%%%%%%%%%%%%%%%%%%%%%%%%%%%%%%%%%%%%%%%%%%%%%%%%%%%%%%%%%%%
The target field and domain for this numerical test are defined as follows.  Let the potential for the field be 
\begin{align}
\dfpot^{(1)}(\vx) = -2g(\tfrac{27}{2}\|\vx\|^4) - \frac12 g(27\|\vx\|^2) - 2\sum_{j=0}^{4} g(9\|\vx-\vxi_j\|^2),
\label{eq:pot_star}
\end{align}
where $\vxi_j = (\cos(2\pi j/5+0.1),\sin(2\pi j/5+\frac12))$ and
\begin{align}
g(r) = \exp(r)/(1 + \exp(r))^2.
\label{eq:generator}
\end{align}  
The target domain is set from the potential as $\Omega^{(1)}=\{\vx\in\R^2|\dfpot^{(1)}(\vx) \leq -\frac{1}{10}\}$, and target div-free vector field is $\vu_{\rm div}^{(1)} = \rot \dfpot^{(1)}$.  This gives a star-like domain with a non-trivial field that is tangential to $\partial\Omega$; see Figure \ref{fig:pot_field_star} for a visualization of the potential and field.  
\begin{figure}[h]
\centering
\begin{tabular}{cc}
\includegraphics[width=0.43\textwidth]{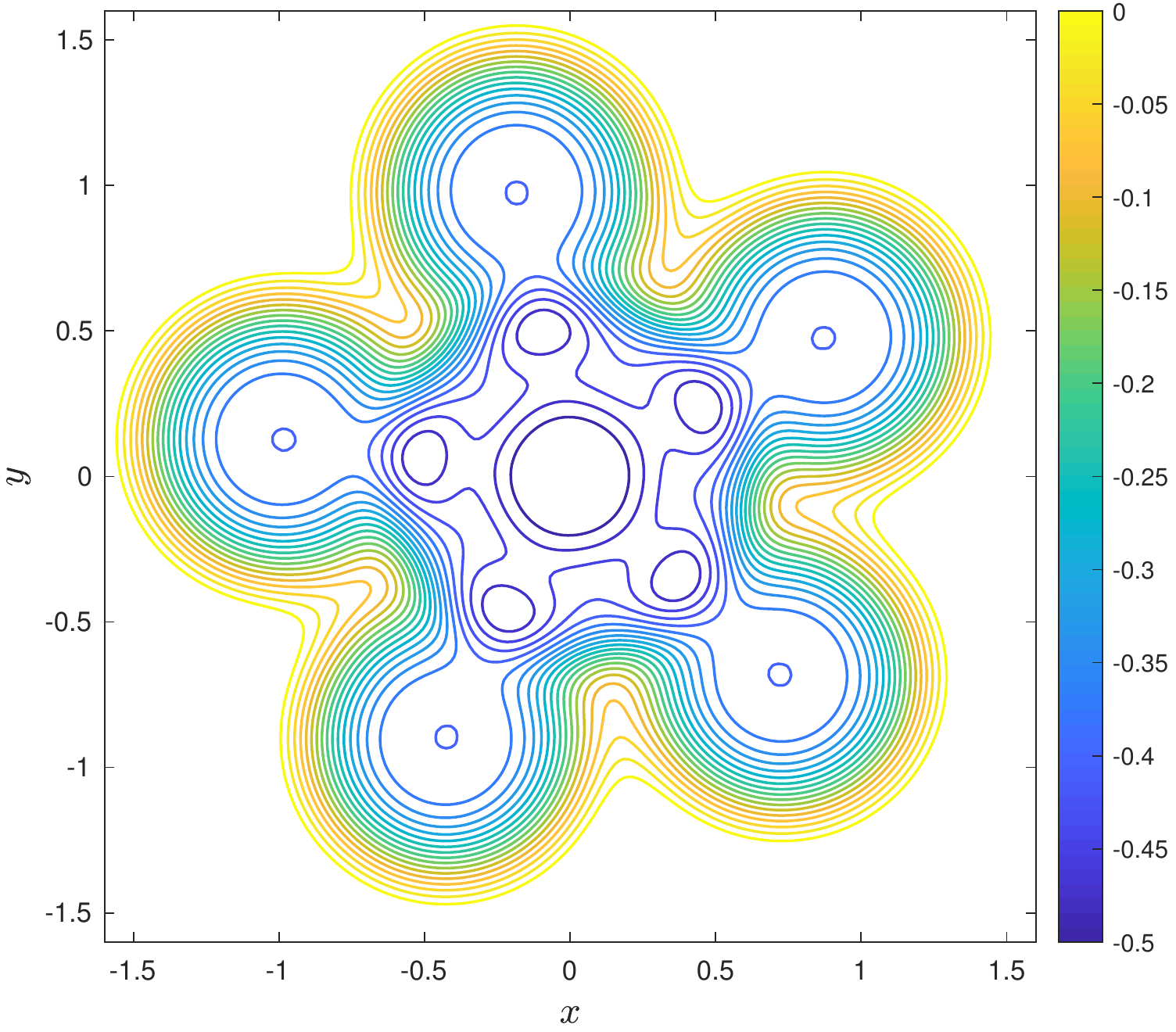} & 
\includegraphics[width=0.38\textwidth]{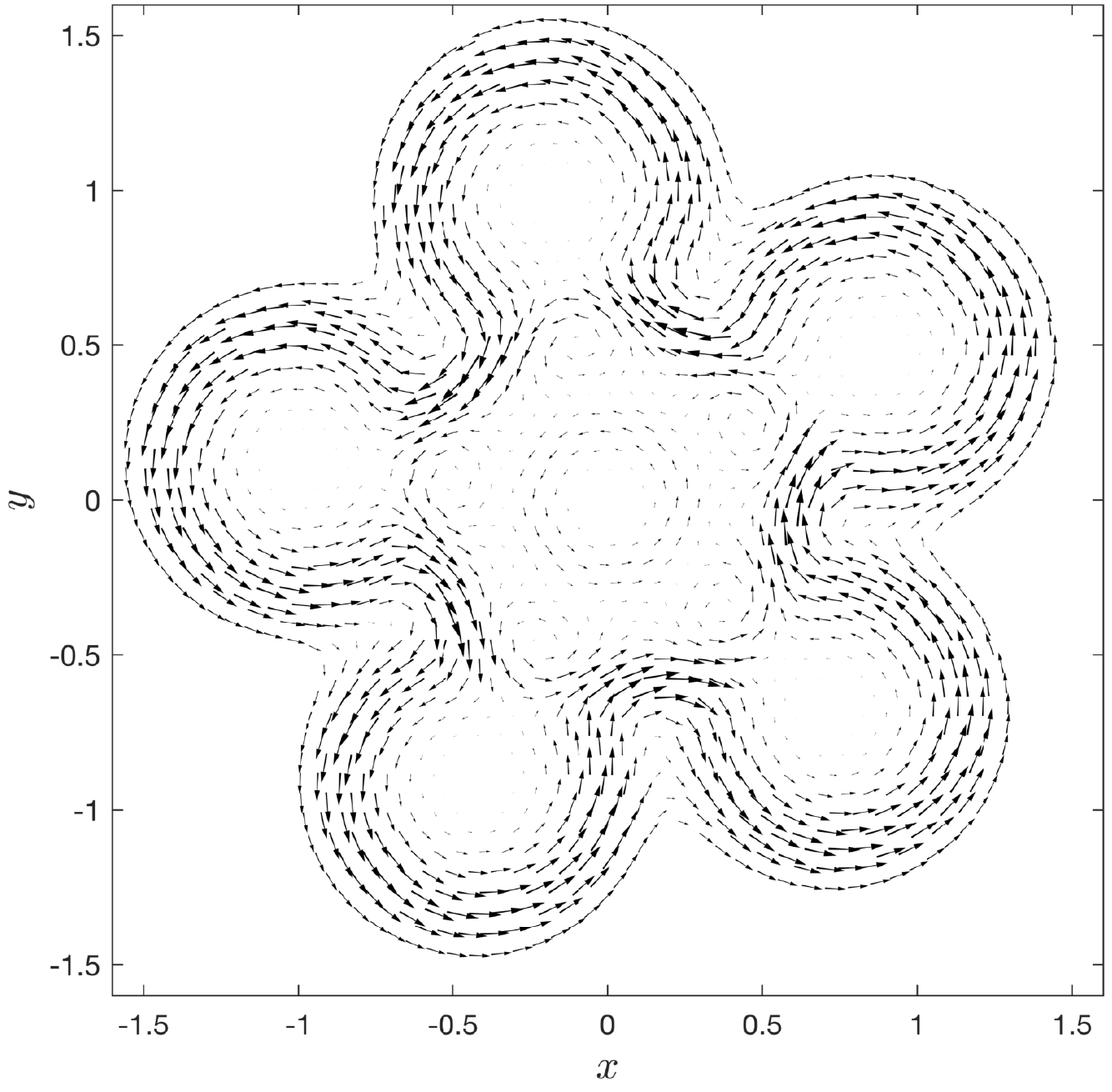}
\end{tabular}
\caption{Contours of the potential $\dfpot^{(1)}$ (left) and corresponding div-free velocity field $\vu_{\rm div}^{(1)}$ (right) for the numerical experiment on $\R^2$.}
\label{fig:pot_field_star}
\end{figure}

\begin{figure}[htb]
\centering
\begin{tabular}{cc}
\includegraphics[width=0.45\textwidth]{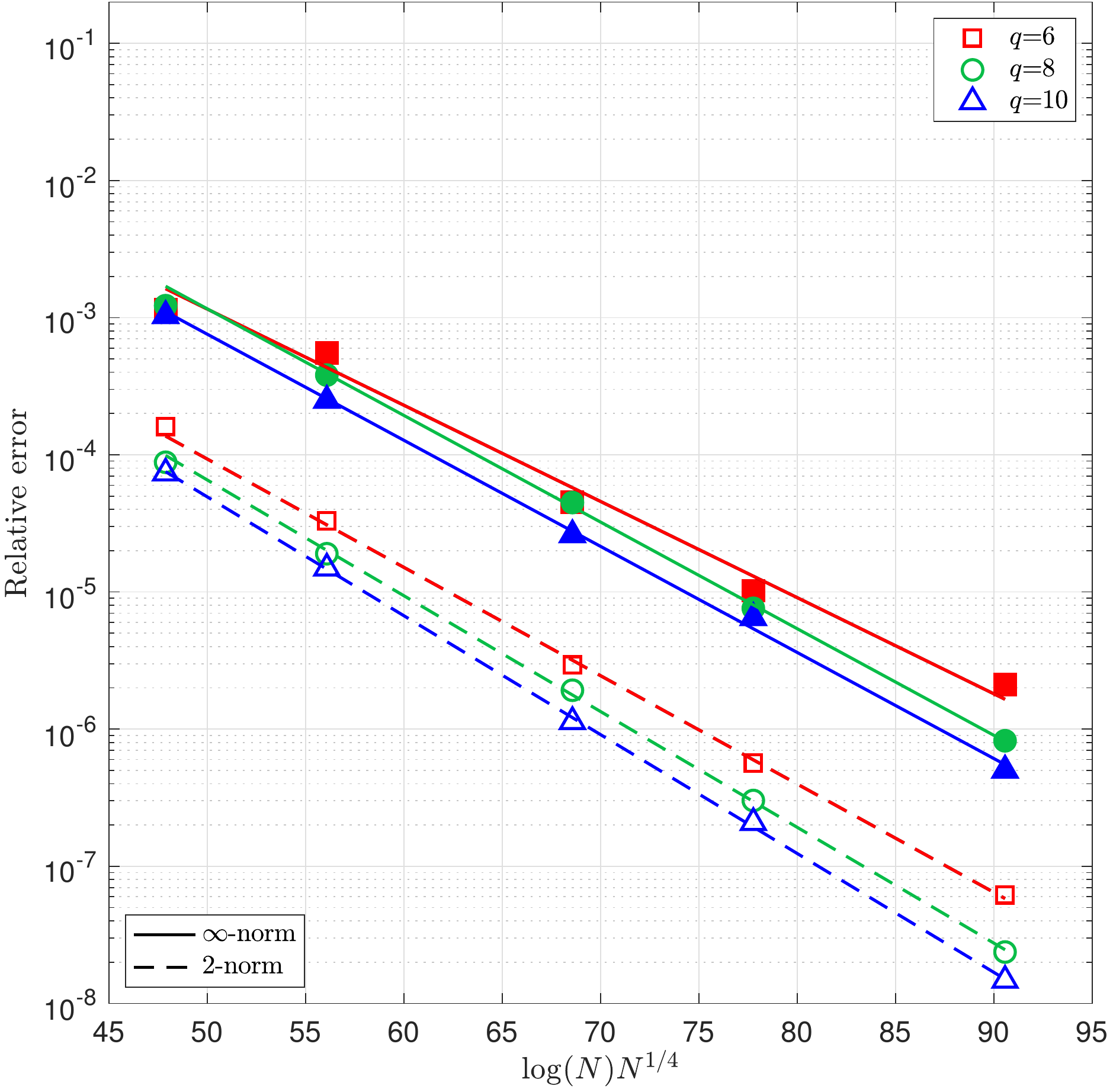} & 
\includegraphics[width=0.45\textwidth]{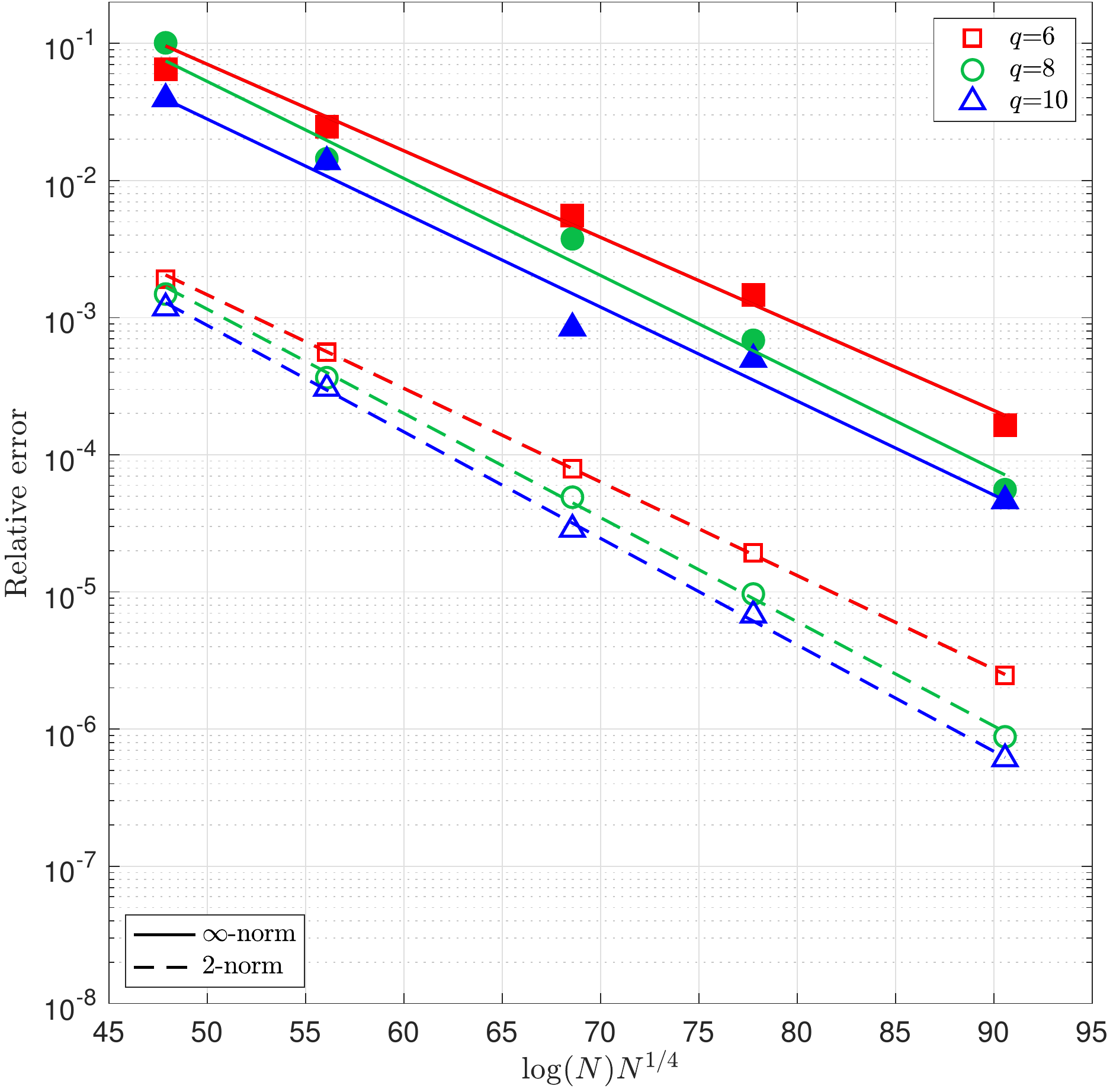} \\
(a) Errors for the potential $\dfpot^{(1)}$ & (b) Errors for the field $\vu_{\rm div}^{(1)}$
\end{tabular}
\caption{Convergence results for the numerical experiment on the star domain in $\R^2$ for the IMQ kernel and different values of $q$.  Filled (open) markers correspond to the relative $\infty$-norm (2-norm) errors and solid (dashed) lines indicate the fit to the estimate $\mathcal{E}(N) = e^{-C\log(N)N^{1/4}}$, without the first values included.}
\label{fig:convg_star}
\end{figure}

The node sets $X$ for this test were initially generated from DistMesh~\cite{Persson_distmeshpaper}, but then perturbed by a small amount to remove any regular structures.  The sizes of the node sets for the tests are $N=11149$, $17405$, $30943$, $44570$, and $69635$\footnote{\revision{These node sets were produced from DistMesh when setting the ``spacing'' parameter to $h0=0.025,0.02,0.015,0.0125,0.01$}}.  We estimate $A$ in \eqref{eq:H_heuristic} to be 6, and  use an overlap parameter for the patches of $\delta = 1/2$.  We test three different values of $q$ to see how the errors are effected by increasing the nodes per patch.  For $q=6,8,10$, there are an average of $63,112,173$ nodes per patch, respectively.  The boundaries create some variability in the nodes per patch and lead to minimum values of $32,57,85$ and the maximums of $109,191,300$, respectively.  As mentioned above, we only report results for the IMQ kernel, for which the shape parameter is set to $\ep=13$ for all tests.  Errors in the approximations of the target potential and field are computed at a dense set of 94252 points over the domain.  Errors in the approximation of the target potential are computed after first normalizing the approximant and the potential to have a mean of zero over the evaluation points.  For each $N$ and $q$, the error reported is the average of the $\infty$-norm ($2$-norm) errors using 20 different random perturbations of the initial node set $X$.  \revision{This reduces fluctuations in the errors caused by particularly good samples of the target field.  We observed that the relative standard deviation in the norms of the errors using this sampling technique varied from 5\% to 10\% for the 2-norm and 20\% to 40\% for the $\infty$-norm across the $N$ we used.}

Figure \ref{fig:convg_star} displays the relative $\infty$-norm and $2$-norm errors in the approximation of the target potential and field as a function of $\log(N)N^{1/4}$.  Included in the figures are the lines of best fit to the errors using the error estimate $\mathcal{E}(N) = e^{-C\log(N)N^{1/4}}$ from scalar RBF theory.  We see from the figure that this error estimate provides a good fit to both the $\infty$-norm and $2$-norm errors for the potential and the field.  The $\infty$-norm errors for the potential have more variability especially for $q=6$, but the $2$-norm errors are quite consistent.  As expected, the errors in reconstructing the potential are lower than those for reconstructing the field, and the $2$-norm errors are lower than the $\infty$-norm errors.  Increasing $q$ leads to a consistent decrease in the $2$-norm errors, but the decrease is more variable for the $\infty$-norm errors.

%The slopes of the lines fitting the potential errors are fairly consistent across the values of $q$ tested, but 

%\begin{figure}[h]
%\centering
%\begin{tabular}{cc}
%\includegraphics[width=0.45\textwidth]{err_vs_N_divfree_matern_plane} & 
%\includegraphics[width=0.45\textwidth]{err_vs_N_divfree_imq_plane} \\
%(a) Fixed $H$ Mat\'ern (log-log scale) & (b) Fixed $H$ IMQ (log-linear scale) \\
%\includegraphics[width=0.45\textwidth]{err_vs_N_varH_divfree_matern_plane} & 
%\includegraphics[width=0.45\textwidth]{err_vs_N_varH_divfree_imq_plane} \\
%(b) Variable $H$ Mat\'ern (log-log scale) & (c) Variable $H$ IMQ (log-linear scale) 
%\end{tabular}
%\caption{Convergence rates for the numerical experiment on the star domain in $\R^2$ for two different kernels.  Filled (open) markers correspond to the errors for the field (potential).  Solid (dashed) lines indicate the lines of best fit to field (potential) errors, without the first values included. For (a) the fits are done using the log of the errors vs.\ log of $\sqrt{N}$, while for (b) the fits are done using the log of the errors vs.\ $\sqrt{N}$.  In the former case, the slopes of these lines are included in the legend for each $H$, with the first number for the field and the second for the potential.}
%\label{fig:convg_star}
%\end{figure}

%%%%%%%%%%%%%%%%%%%%%%%%%%%%%%%%%%%%%%%%%%%%%%%%%%%%%%%%%%%%%%%%%%%%
\subsection{Div-free field on $\sphere^2$}\label{numerics_sphere}
%%%%%%%%%%%%%%%%%%%%%%%%%%%%%%%%%%%%%%%%%%%%%%%%%%%%%%%%%%%%%%%%%%%%
Let $\vx=(x,y,z)\in\sphere^2$, and the potential for the target field be defined as
\begin{align}
\dfpot^{(2)}(\vx) = -\frac{1}{1 + e^{-20(z+1/\sqrt{2})}} - \frac{1}{1 + e^{-20(z-1/\sqrt{2})}} - 3\sum_{j=0}^5 (-1)^j  g(\|\vx-\vy_j\|^2,a_j),
\label{eq:pot_star}
\end{align}
where $g$ is given in \eqref{eq:generator}, $\vy_j = (\cos(\lambda_j)\cos(\theta_j),\sin(\lambda_j)\cos(\theta_j),\sin(\theta_j))$ for $\{\lambda_j\}_{j=0}^5=$ $\{0.05$,$1.1$,$2.12$, $3.18$,$4.22$,$5.26\}$ and  $\{\theta_j\}_{j=0}^5=$$\{0.79,-0.82$,$0.76$,$-0.81$,$0.8$,$-0.77\}$, and $a_j = 4 + j/2$.  The div-free field is then given as $\vu_{\rm div}^{(2)} = \rot{\dfpot^{(2)}}$.  The values used in \eqref{eq:pot_star} were chosen to produce a zonal jet in the mid-latitudes with three superimposed vortices in each of the northern and southern hemispheres; see Figure \ref{fig:pot_field_sphere} for a visualization of the potential and field.

\begin{figure}[h]
\centering
\begin{tabular}{cc}
\includegraphics[width=0.41\textwidth]{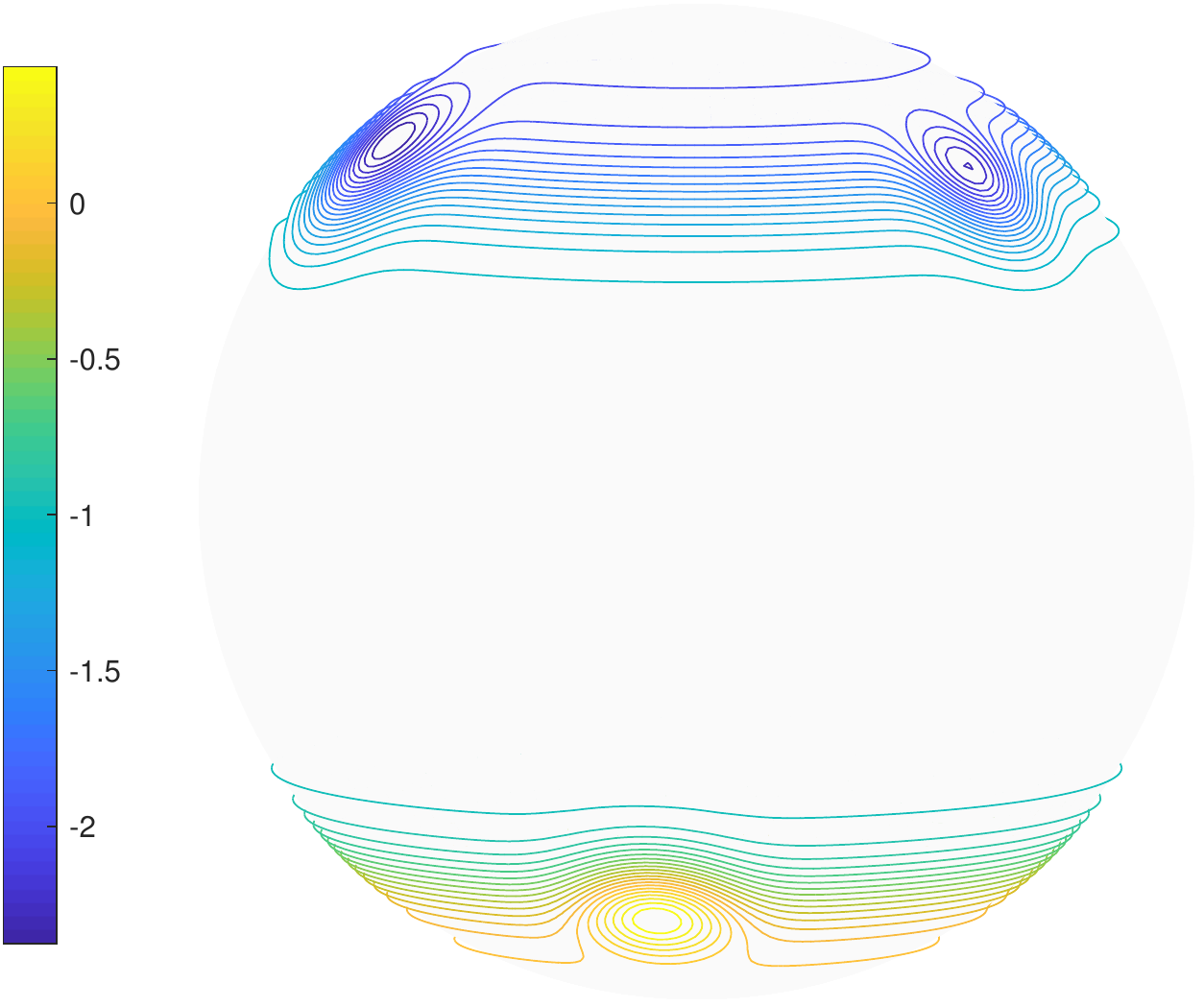} & 
\includegraphics[width=0.35\textwidth]{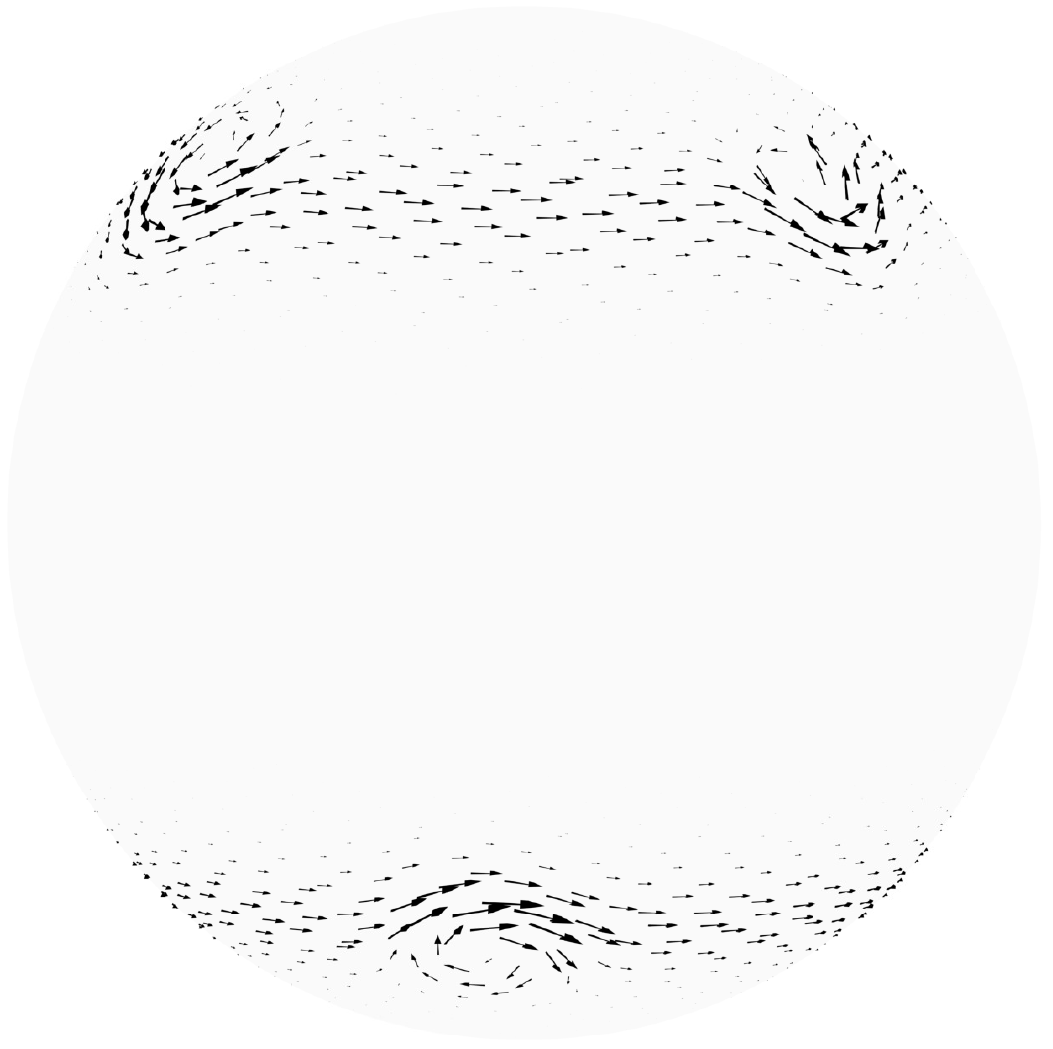}
\end{tabular}
\caption{Contours of the potential $\dfpot^{(2)}$ (left) and corresponding div-free velocity field $\vu_{\rm div}^{(2)} $(right) for the numerical experiment on $\sphere^2$.}
\label{fig:pot_field_sphere}
\end{figure}
%\begin{figure}[h]
%\centering
%\begin{tabular}{cc}
%\includegraphics[width=0.45\textwidth]{errs_vs_N_divfree_potential_imq_sphere} & 
%\includegraphics[width=0.45\textwidth]{errs_vs_N_divfree_field_imq_sphere} \\
%(a) IMQ Errors for the potential $\dfpot^{(2)}$ & (b) IMQ Errors for the field $\vu_{\rm div}^{(2)}$ \\
%\includegraphics[width=0.45\textwidth]{errs_vs_N_divfree_potential_matern_sphere} & 
%\includegraphics[width=0.45\textwidth]{errs_vs_N_divfree_field_matern_sphere} \\
%(c) Mat\'ern Errors for the potential $\dfpot^{(2)}$ & (d) Mat\'ern Errors for the field $\vu_{\rm div}^{(2)}$
%\end{tabular}
%\caption{Convergence results for the numerical experiment on $\sphere^2$ for the IMQ kernel and different values of $q$.  Filled (open) markers correspond to the relative $\infty$-norm (2-norm) errors and solid (dashed) lines indicate the fit to the expected error estimate $\mathcal{E}(N) = e^{-C\log(N)N^{1/4}}$, without the first values included.}
%\label{fig:convg_sphere}
%\end{figure}
\begin{figure}[h]
\centering
\begin{tabular}{cc}
\includegraphics[width=0.45\textwidth]{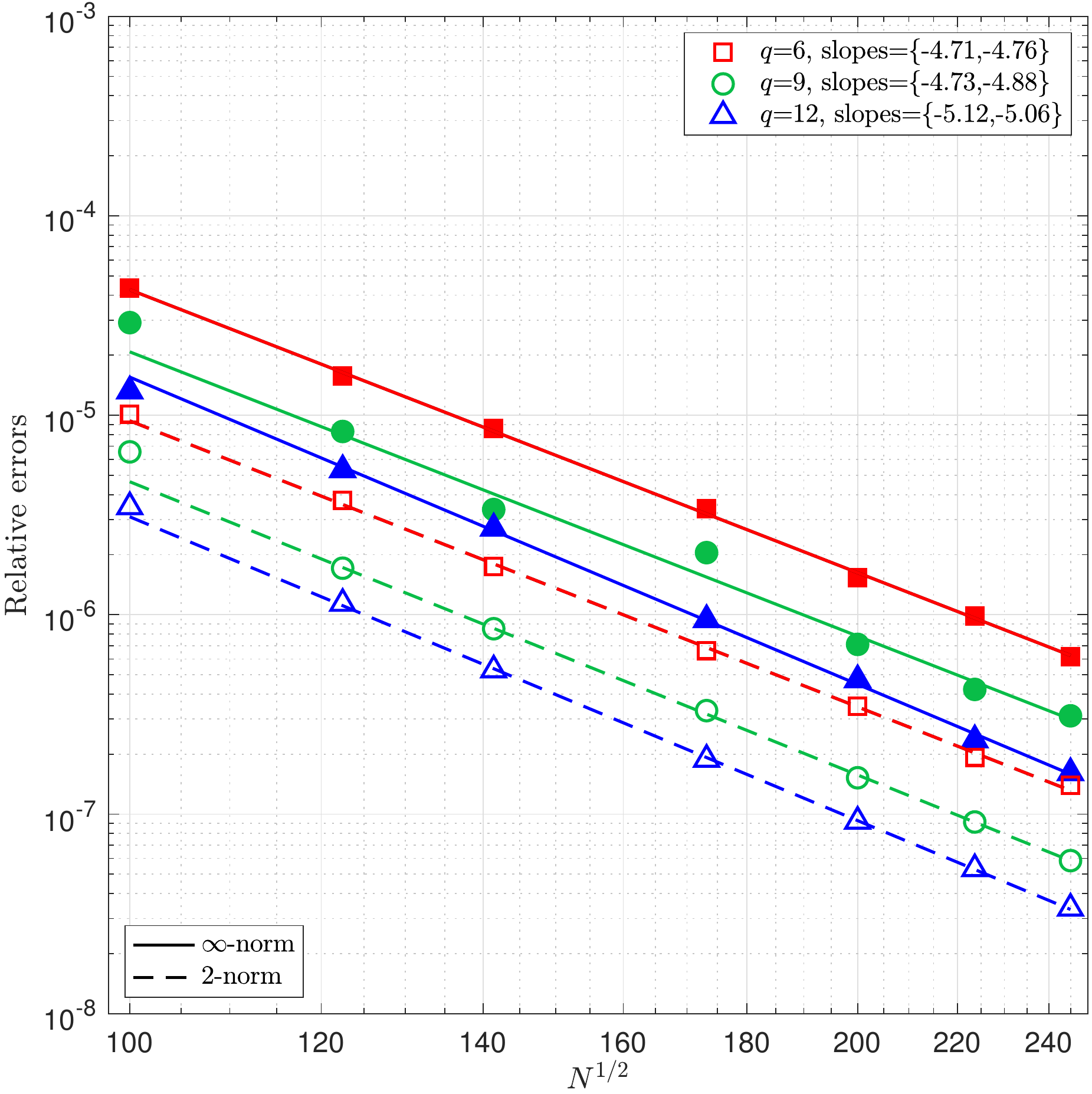} & 
\includegraphics[width=0.45\textwidth]{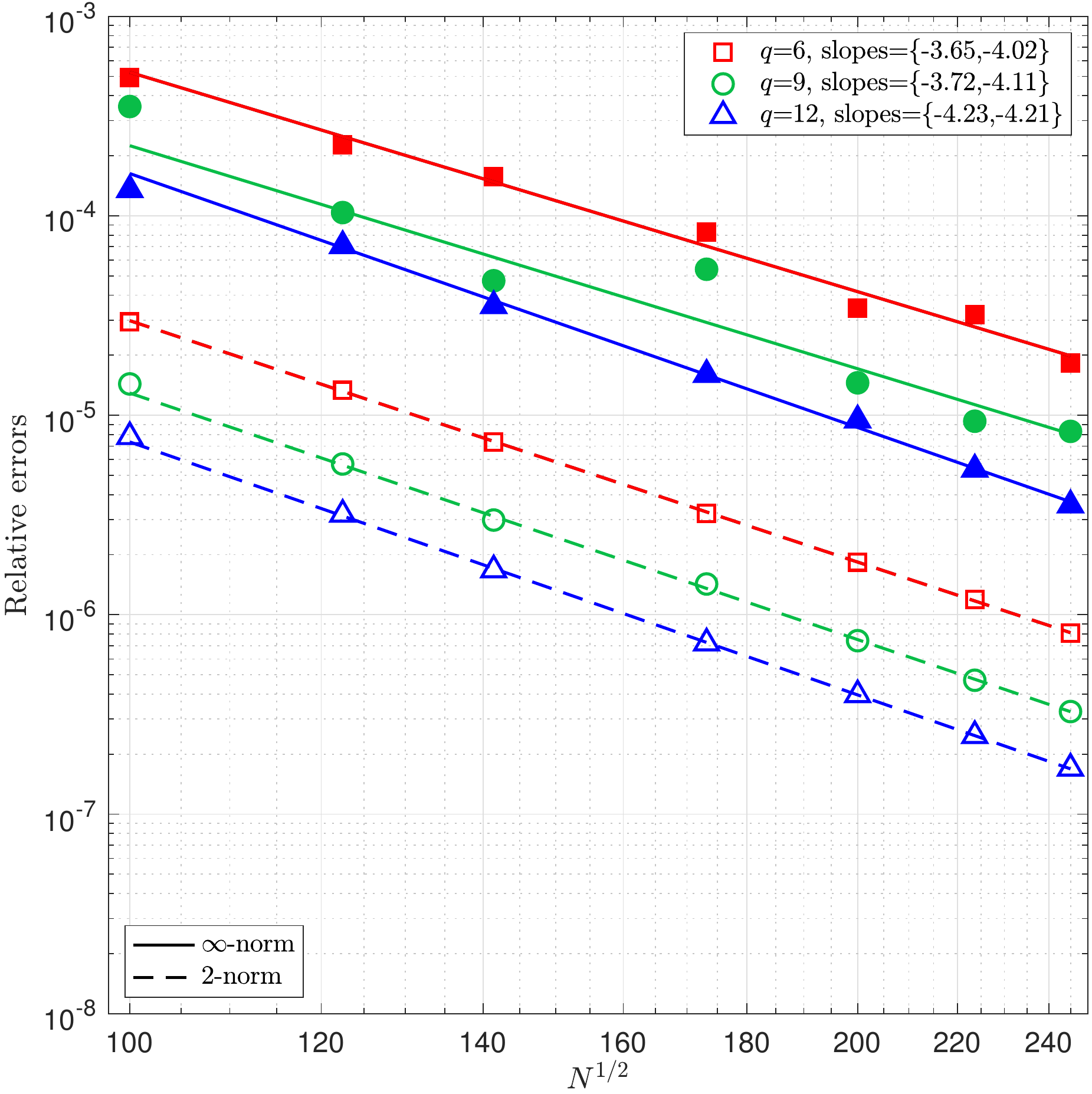} \\
(c) Errors for the potential $\dfpot^{(2)}$ & (d) Errors for the field $\vu_{\rm div}^{(2)}$
\end{tabular}
\caption{Convergence rates for the numerical experiment on $\sphere^2$ for the Mat\'ern kernel and different values of $q$.  Filled (open) markers correspond to the relative $\infty$-norm (2-norm) errors and solid (dashed) lines indicate the lines of best fit to the $\infty$-norm (2-norm) errors as a function of $\sqrt{N}$ on a loglog scale.  The legend indicates the slopes of these lines with the first number corresponding to the $\infty$-norm and the second the 2-norm, which give estimates for the algebraic convergence rates.}
\label{fig:convg_sphere}
\end{figure}

The node sets $X$ for this test are chosen as Hammersley nodes, which give quasiuniform, but random sampling points for $\sphere^2$~\cite{SpherePts}.  The sizes of the node sets for the tests are $N=10000$, $15000$, $20000$, $30000$, $40000$, $50000$ and $60000$.  We use $A=4\pi$ in \eqref{eq:H_heuristic} and set the overlap parameter to $\delta = 9/16$.  We again use three different values of $q$ to see how the errors are effected by increasing the nodes per patch.  For $q=6,9,12$, there are an average of $63,143,252$ nodes per patch, respectively.  Since there are no boundaries for this domain, the number of nodes per patch is much more consistent across all patches.  The minimum nodes per patch are $58,137,245$ and the maximums are $69,150,261$, respective to the $q$ values.  For this example, we only report results for the Mat\'ern kernel, for which the shape parameter is set to $\ep=7.5$ for all tests.  Errors in the approximations of the target potential and field are computed at a quasiuniform set of 92163 points over $\sphere^2$.  Errors in the approximation of the target potential are again computed after first normalizing the approximant and the potential to have a mean of zero over the evaluation points. Similar to the previous experiment, for each $N$ and $q$, the error reported is the average of the $\infty$-norm ($2$-norm) errors from 20 different random rotations of the initial Hammersley node set $X$.  \revision{We observed similar results on the relative standard deviations of the norms of the errors as the previous experiment using this sampling technique.}

\begin{figure}[h]
\centering
\includegraphics[width=0.75\textwidth]{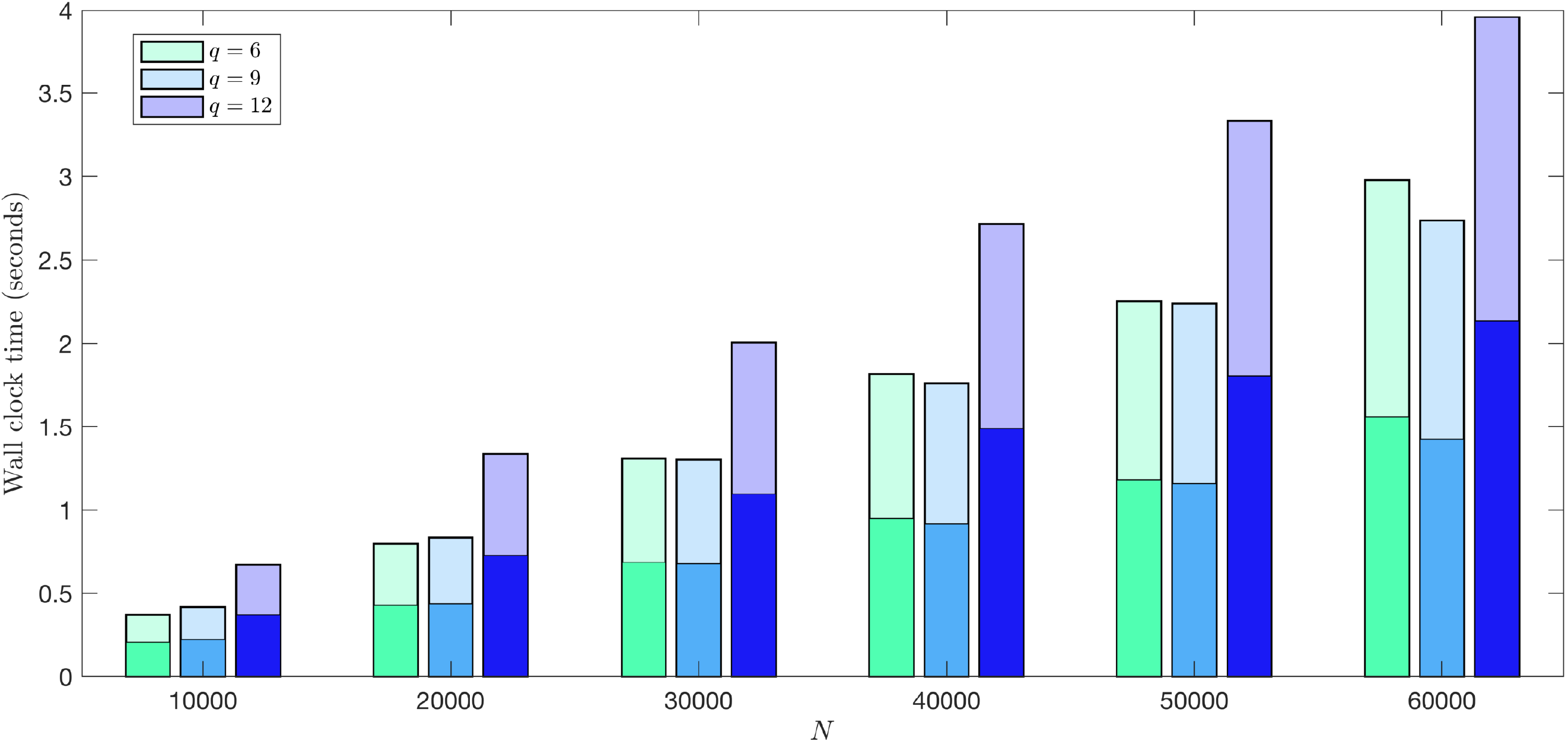}
\caption{\revision{Timing results for the numerical experiment on $\sphere^2$ with different values of $q$.  The darker region of each bar marks the time it takes to compute the interpolation coefficients on each patch and solve for the potential shifts, while the full bar includes this time and the time it takes to evaluate the approximant of the field and the potential at $N$ points.}}
%\caption{Timing results for the numerical experiment on $\sphere^2$ with different values of $q$.  The dashed lines are the lines of best fit to the timings using all but the first two values.}
\label{fig:timing_sphere}
\end{figure}

Figure \ref{fig:convg_sphere} displays the relative $\infty$-norm and $2$-norm errors in the approximation of the target potential and field as a function of $N^{1/2}$.  Included in the figure are the lines of best fit to the log of the errors vs.\ the log of $N^{1/2}$ for each $q$, and the slopes of these lines are reported in the legend of the figure (where the first number is for $\infty$-norm and second for the $2$-norm).  We see from this figure that the computed rates of convergence for the $\infty$-norm are slightly higher than the theoretical rate of $-3.5$.  Thus the residual estimate from Proposition \ref{residualbound} is not leading to a reduction in the convergence rates as discussed at the end of Section \ref{mainerrorsection}.   We also see from the figure that the estimated rates for the $2$-norm errors are higher than the $\infty$-norm errors as one would expect.  Finally, similar to the previous experiment, we see that the errors in reconstructing the potential are lower than those for reconstructing the field.

\revision{We also display timing results for this experiment in Figure \ref{fig:timing_sphere}.  For these results, we scaled the evaluation points with $N$ and measured the time for the fitting phase of the method (determining the interpolation coefficients on each patch and the potential shifts) and the evaluation phase (evaluating the approximants of the field and potential on each patch and combining these using the PU weight functions).  The results for $q=9$ and $q=12$ show a clear linear scaling with $N$, but the rate appears to be a bit higher for $q=6$, which we anticipate is due to not being in the asymptotic range of $N$ for this case.  Also, the predicted $\bigO(N\log N)$ complexity is most likely not visible over the range of $N$ considered.  In all the results, we see that the evaluation phase takes less time than the fitting phase, which is expected since the cost for this phase is $\bigO(n^2)$ per patch vs.\ $\bigO(n^3)$ for fitting.  Interestingly, with this serial version of the code, $q=9$ is overall the fastest.  Since the number of patches is inversely proportional to $q^2$, these results indicate that there is an optimal value that balances solving fewer larger systems to more smaller systems.}

\subsection{Curl-free field on the unit ball}\label{numerics_r3}
%%%%%%%%%%%%%%%%%%%%%%%%%%%%%%%%%%%%%%%%%%%%%%%%%%%%%%%%%%%%%%%%%%%%
The target curl-free field for this test is generated as follows.  Let $g(r,a) = (a + r^2)^{-1/2}$ and define the following potential:
\begin{align}
\dfpot^{(3)}(\vx) =  -\frac14 g(\|\vx\|,0.1) + \frac18 \sum_{j=1}^12 g(\|\vx-\vxi_j\|,0.04),
\label{eq:pot_ball}
\end{align}
where $\{\vx_j\}_{j=1}^{12}$ are the vertices of a regular icosahedron with each vertex a distance of 2/3 from the origin. The target curl-free is then generated by $\vu_{\rm curl}^{(3)} = -\nabla \dfpot^{(3)}$.   This field can be interpreted as the (idealized) electric field that is generated from a negative (smoothed) point charge at the origin, surrounded by 12 positive (smoothed) point charges, equidistance from one another; see Figure \ref{fig:nodes_pot_ball}(a) for a visualization of the potential and field.

\begin{figure}[h]
\centering
\begin{tabular}{cc}
\includegraphics[width=0.4\textwidth]{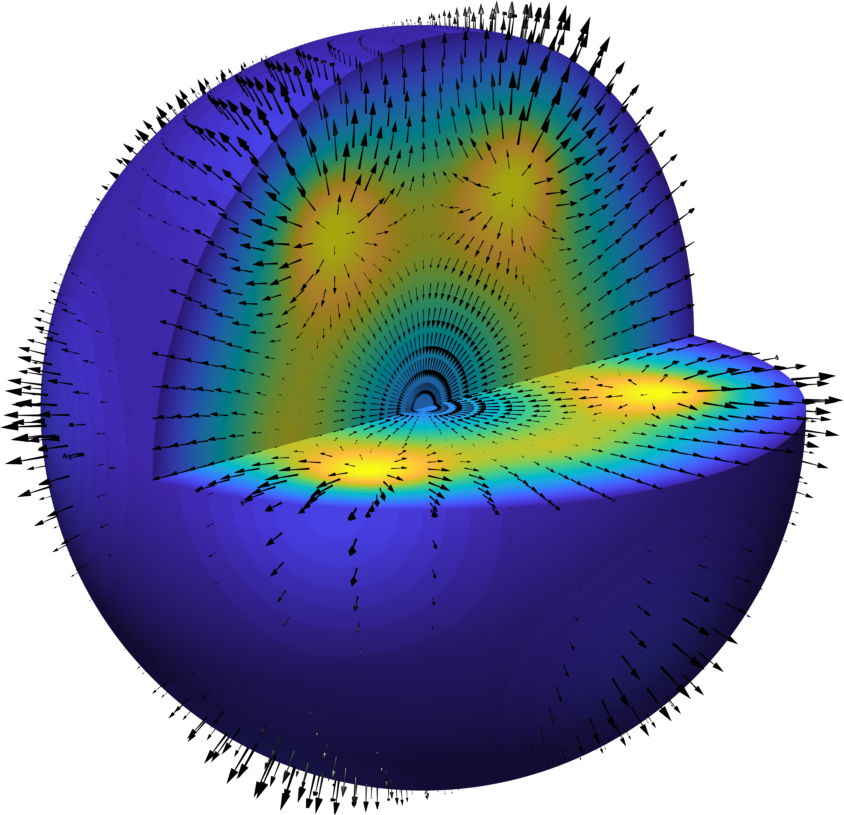} & \includegraphics[width=0.37\textwidth]{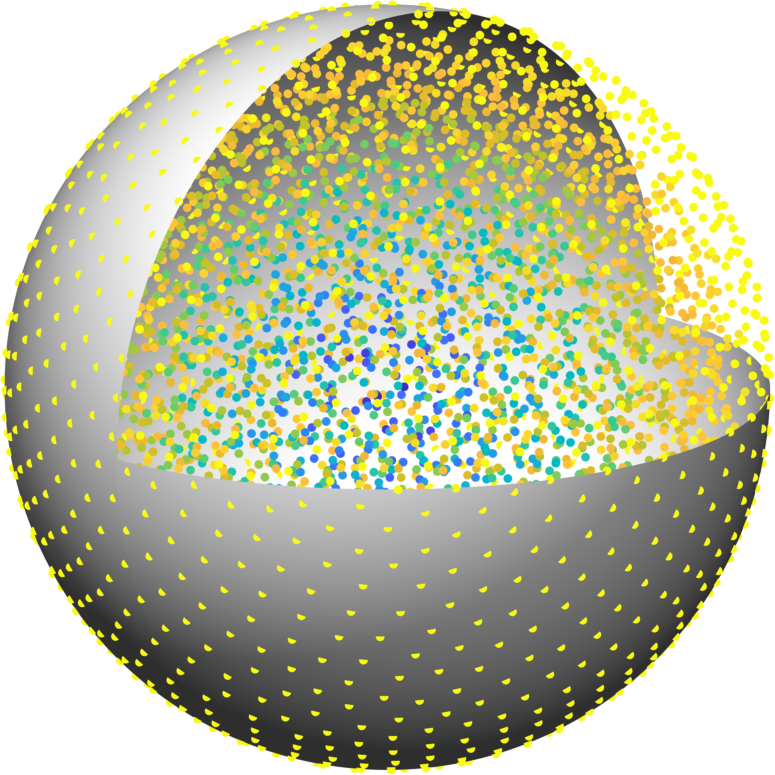}\\
(a) Potential and field & (b) Nodes
\end{tabular}
\caption{(a) Visualization of the potential $\cfpot^{(3)}$ and corresponding curl-free velocity field $\vu_{\rm curl}^{(3)} = -\nabla \cfpot^{(3)}$ for the numerical experiment on the unit ball.  (b) Example of $N=4999$ node set (small solid disks) used in the numerical experiment on the unit ball, where colors of the nodes are proportional to their distance from the origin (yellow=1, green = 0.5, blue=0).  The plots in both figures show the unit ball with a wedge removed to aid in the visualization.}
\label{fig:nodes_pot_ball}
\end{figure}

%\begin{figure}[h]
%\centering
%\begin{tabular}{cc}
%\includegraphics[width=0.45\textwidth]{err_vs_N_curlfree_matern_ball} & 
%\includegraphics[width=0.45\textwidth]{err_vs_N_curlfree_IMQ_ball} \\
%(a) Fixed $H$ Mat\'ern (log-log scale) & (b) Fixed $H$ IMQ (log-linear scale) \\
%\includegraphics[width=0.45\textwidth]{err_vs_N_varH_curlfree_matern_ball} & 
%\includegraphics[width=0.45\textwidth]{err_vs_N_varH_curlfree_imq_ball} \\
%(b) Variable $H$ Mat\'ern (log-log scale) & (c) Variable $H$ IMQ (log-linear scale) 
%\end{tabular}
%\caption{Convergence rates for the numerical experiment on the unit ball $\mathbb{D}^3$ for two different kernels.  Filled (open) markers correspond to the errors for the field (potential).  Solid (dashed) lines indicate the lines of best fit to field (potential) errors, without the first values included. For (a) the fits are done using the log of the errors vs.\ log of $\sqrt{N}$, while for (b) the fits are done using the log of the errors vs.\ $\sqrt{N}$.  In the former case, the slopes of these lines are included in the legend for each $H$, with the first number for the field and the second for the potential.}
%\label{fig:convg_ball}
%\end{figure}

\begin{figure}[h]
\centering
\begin{tabular}{cc}
\includegraphics[width=0.45\textwidth]{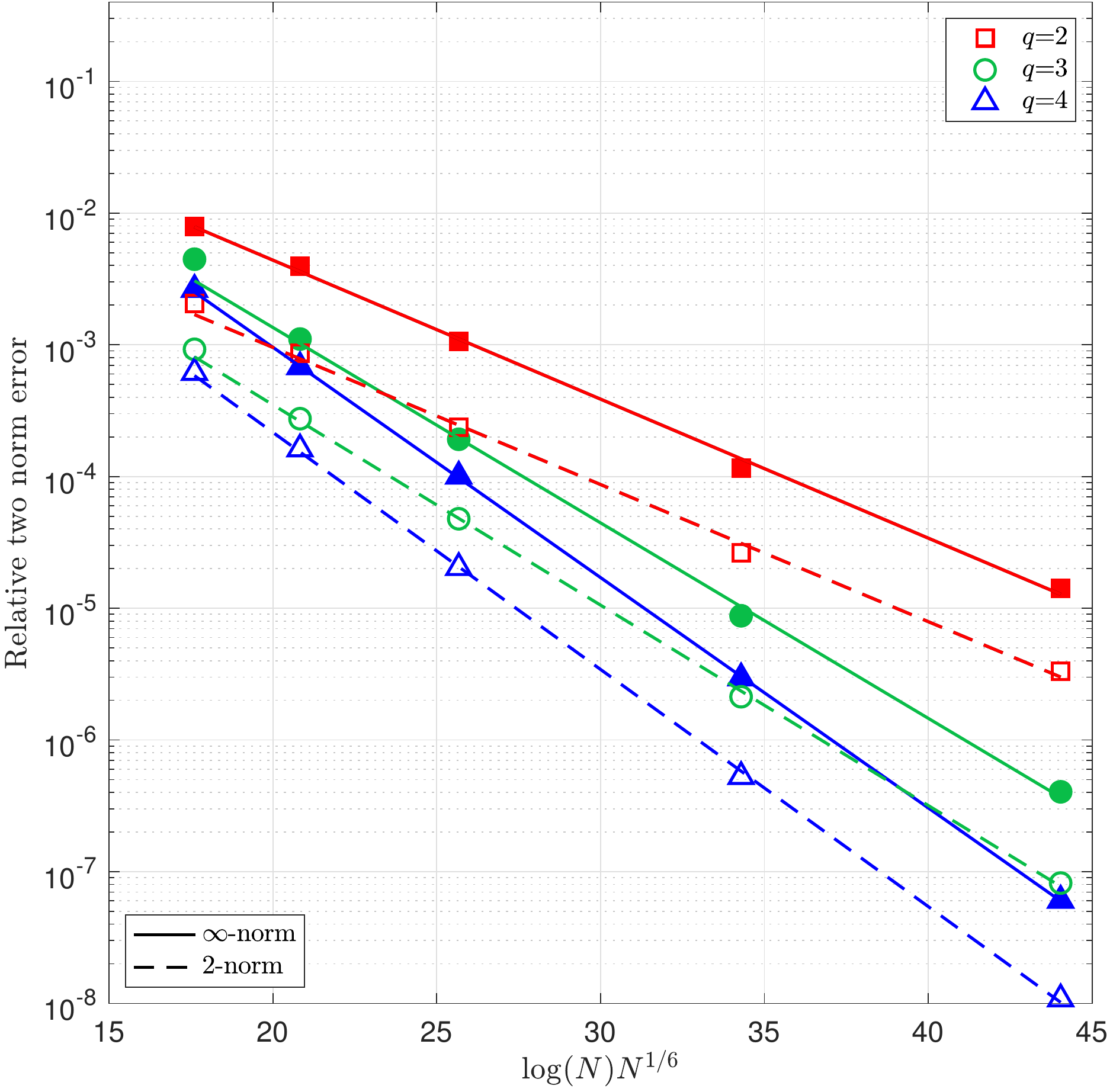} & 
\includegraphics[width=0.45\textwidth]{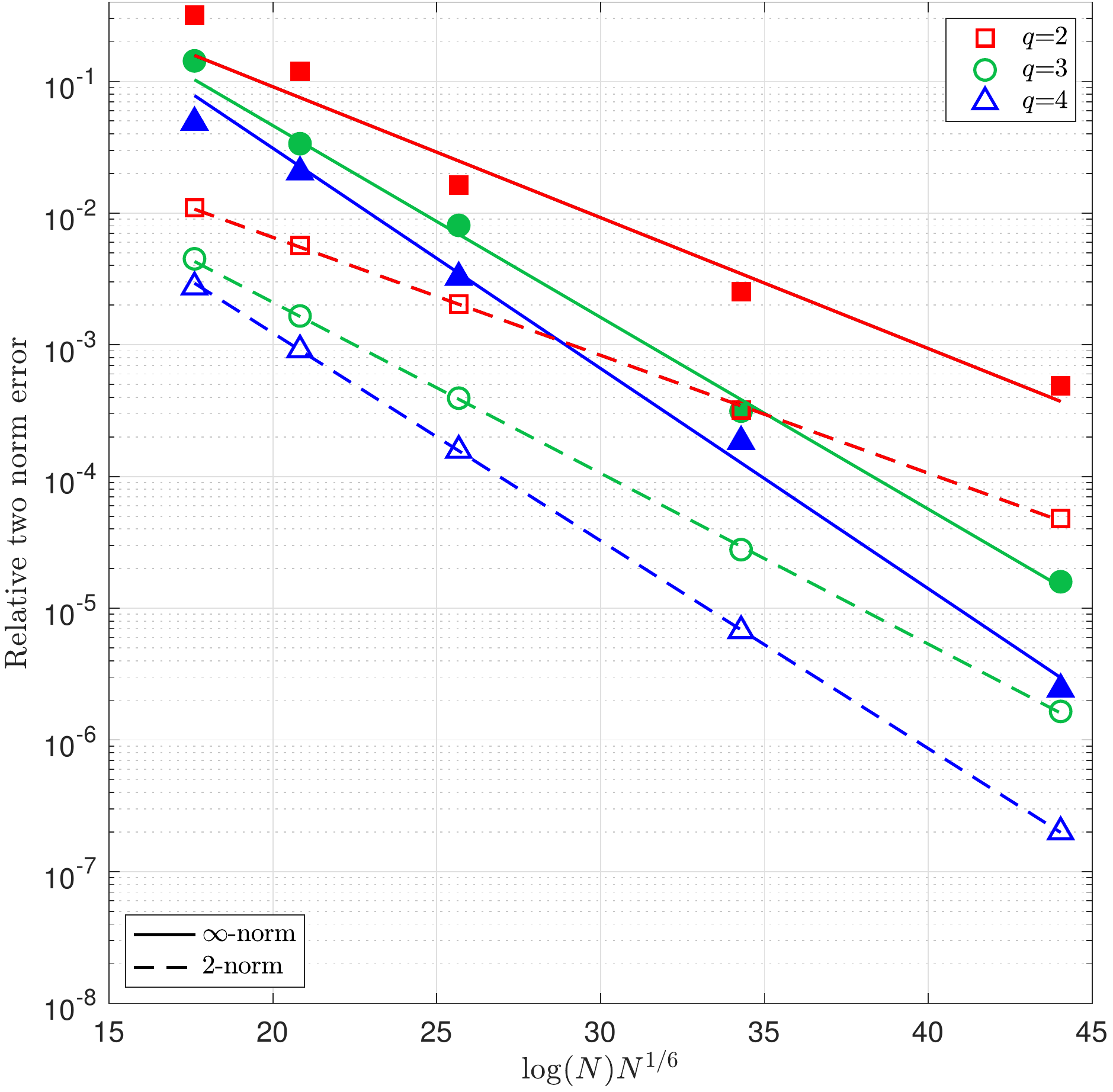} \\
(a) Errors for the potential $\cfpot^{(3)}$ & (b) Errors for the field $\vu_{\rm curl}^{(3)}$
\end{tabular}
\caption{Convergence results for the numerical experiment on the unit ball in $\R^3$ for the IMQ kernel and different values of $q$.  Filled (open) markers correspond to the relative $\infty$-norm (2-norm) errors and solid (dashed) lines indicate the fit to the expected error estimate $\mathcal{E}(N) = e^{-C\log(N)N^{1/6}}$, without the first values included.}
\label{fig:convg_ball}
\end{figure}

The node sets $X$ for this test are obtained from the meshfree node generator described in~\cite{ShankarKirbyFogelson}, which produces quasiuniform but unstructured nodes in general domains; see Figure \ref{fig:nodes_pot_ball} (b) for an example of the nodes used for the unit ball. The sizes of the node sets for the tests are $N=4999$, $9103$, $19636$, $59116$, and $158474$\footnote{\revision{These node sets were produced from the node generator~\cite{ShankarKirbyFogelson} when setting the ``spacing'' parameter to $h0=0.1, 0.08,0.06,0.04,0.028$}}.  We use $A=4/3\pi$ in \eqref{eq:H_heuristic} and an overlap parameter of $\delta = 1/4$.  We again test three different values of $q$: $q=2,3,4$.  For $q=2$, the minimum, average, and maximum nodes per patch are $18$, $37$, $83$, for $q=3$ these values are $72$, $120$, $238$, and for $q=4$ these values are $186$, $271$, $512$.  As with the first experiment, we only present results for the IMQ kernel, for which the shape parameter is set to $\ep=4$ for all tests.  Errors in the approximations of the target potential and field are computed at a set of 208707 points over the unit ball.  Errors in the approximation of the target potential are again computed after first normalizing the approximant and the potential to have a mean of zero over the evaluation points. Similar to the previous experiments, for each $N$ and $q$, the error reported is the average of the $\infty$-norm ($2$-norm) errors from 20 different random rotations of the initial node set $X$.

Figure \ref{fig:convg_ball} displays the relative $\infty$-norm and $2$-norm errors in the approximation of the target potential and field as a function of $\log(N)N^{1/6}$.  As in the first experiment, we have included the lines of best fit to the errors, but now using $\mathcal{E}(N) = e^{-C\log(N)N^{1/6}}$.  We see from the Figure that the error estimate again generally provides a good fit to both the $\infty$-norm and $2$-norm errors for the potential and the field.  The $\infty$-norm errors deviate more from the estimates than the $2$-norm errors, especially for field in the $q=2$ case.  However, for this case the minimum number of points per patch can be quite small. 

\begin{remark}
\revision{In practice, there are several parameters a user needs to choose in the algorithm that effect the computational cost and accuracy.  In the experiments reported here, and several others not reported, we have explored these parameters and come up with the following suggestions.  For the $q$ parameter, which controls the average nodes per patch, we recommend a value in the range of $8\leq q \leq 9$ for 2D problems and $3\leq q \leq 4$ for 3D problems.  For the overlap parameter, $\delta$, we recommend a value in the range $1/2 \leq \delta \leq 3/4$.  For the shape parameter $\ep$, we recommend choosing it as small as possible on each patch before ill-conditioning sets in when solving the local linear systems~\eqref{eq:divfree_linsys}.  This is similar to the method~\cite{SWFKJSC2014} used for generating RBF finite difference formulas.  For smooth vector fields, this typically gives the best accuracy for a given $N$.}
\end{remark}

%%%%%%%%%%%%%%%%%%%%%%%%%%%%%%%%%%%%%%%%%%%%%%%%%%%%%%%%%%%%
\section{Concluding remarks}
We have presented a new method based on div/curl-free RBFs and PUM for approximating div/curl-free vector fields in $\R^2$ and $\sphere^2$, and for curl-free fields in $\R^3$. The method produces approximants that are analytically div/curl-free and also produces an approximant potential for the field at no additional cost. For quasi-uniform samples, we have shown how the parameters can be selected so that the computational complexity of the method is $\bigO(N\log N)$.   We have proved error estimates for the approximants based on local estimates for the div/curl-free interpolants on the PU patches.  We have also demonstrated the high-order convergence rates of the method on three different test problems with samples ranging from thousands to hundreds of thousands of nodes---all done on a standard laptop.  

While we have only focused on div/curl-free interpolation over local patches, a future area to explore is to instead use a least squares approach similar to the one used for scalar RBFs in~\cite{Larsson2017}.  Here one can choose fewer centers in the local patches for the div/curl-free RBFs than data samples, a technique referred to as regression splines in the statistics literature~\cite[ch. 19]{MeshFree_Mat}.  This has the benefit of further reducing the cost of the local patch solves for the approximation coefficients and could provide some regularization. Another future area to explore is the adaption of stable algorithms for ``flat'' RBFs~\cite{FLF11,FaMC12} to the div/curl-free RBFs.  These algorithms are especially important in scalar RBF-PUM methods based on smooth RBFs for reaching high accuracies~\cite{Larsson2017}.  Some work has been done along these lines for $\sphere^2$ in~\cite{DRAKE2020109595}, but not for the local setting on patches.  A final promising area for future research is in developing adaptive algorithms for the method along the lines of~\cite{cavoretto_2016}.

%%%%%%%%%%%%%%%%%%%%%%%%%%%%%%%%%%%%%%%%%%%%%%%%%%%%%%%%%%%%%

\section*{Acknowledgments}
We thank Elisabeth Larsson for helpful discussions regarding the PU patch distribution algorithm and Varun Shankar for generating the node sets used for the unit ball example. KPD's work was partially supported by the SMART Scholarship funded by The Under Secretary of Defense-Research and Engineering, National Defense Education Program/BA-1, Basic Research. GBW's work was partially supported by National Science Foundation grant 1717556. 

\bibliographystyle{siamplain}
\bibliography{references}

\end{document}